 \documentclass[final,5p,times,twocolumn]{elsarticle}

\usepackage{graphicx,tools}
\usepackage{amsmath,amssymb,xfrac}

\usepackage[dvipsnames]{xcolor}
\usepackage{array}
\usepackage{multirow}
\usepackage{amsfonts}
\usepackage{caption}
\usepackage{setspace}
\usepackage{breakcites}
\usepackage{lipsum}
\usepackage{float,subfig}
\usepackage{tikz}
\usepackage{stmaryrd}
\usepackage{siunitx}
\usepackage{mathtools}
\usepackage{amsthm}
\usepackage{bm}
\usepackage{hyperref}
\usepackage{subfig}
\hypersetup{
colorlinks=true,
linkcolor=blue,
filecolor=magenta,
urlcolor=cyan,
citecolor=gray,
}
\usepackage{lineno}
\usetikzlibrary{patterns}
\usepackage{pgfplots}
\pgfplotsset{width=7.5cm,compat=1.12}

\newtheorem{theorem}{Theorem}[section]
\newtheorem{corollar}{Corollary}[section]
\newtheorem{lemma}{Lemma}[section]
\newtheorem{proposition}{Proposition}[section]

\newtheorem{remark}{Remark}[section]

\journal{Computers \& Mathematics with Applications \vf{(CAMWA)}}

\begin{document}

\begin{frontmatter}


\title{Improved error estimates of hybridizable interior penalty methods using a variable penalty for highly anisotropic diffusion problems}

\author[add1]{Gr\'egory Etangsale}
\ead{gregory.etangsale@univ-reunion}
\author[add2]{Marwan Fahs}
\ead{fahs@unistra.fr}
\author[add1]{Vincent Fontaine\corref{cor1}}
\ead{vincent.fontaine@univ-reunion}
\cortext[cor1]{Corresponding and Principal author}
\author[add1]{Nalitiana Rajaonison}
\ead{nalitiana.rajaonison@univ-reunion.fr}

\address[add1]{Department of Building and Environmental Sciences, University of La R\'eunion - South Campus, France}
\address[add2]{$Universit\acute{e}$ de Strasbourg, CNRS, ENGEES, LHYGES UMR 7517, F-67000 Strasbourg, France}

\begin{abstract}
In this paper, we derive improved \textit{a priori} error estimates for families of hybridizable interior penalty discontinuous Galerkin (H-IP) methods using a variable penalty for second-order elliptic problems. The strategy is to use a penalization function of the form $\mathcal{O}(1/h^{1+\delta})$, where $h$ denotes the mesh size and $\delta$ is a user-dependent parameter. We then quantify its direct impact on the convergence analysis, namely, the (strong) consistency, discrete coercivity and boundedness (with $h
^{\delta}$-dependency), and we derive updated error estimates for both discrete energy- and $\LL{2}$-norms. The originality of the error analysis relies specifically on the use of conforming interpolants of the exact solution. All theoretical results are supported by numerical evidence.
\end{abstract}


\end{frontmatter}

\section{Introduction}
\label{S:1}
Hybridizable discontinuous Galerkin (HDG) methods were first introduced in the last decade by Cockburn \textit{et al.} \cite{Cockburn09unified} (see, \eg \cite{egger2009mixed}) and have since received extensive attention
from the research community. They are popular and very efficient numerical approaches for solving a large class of \vf{linear and nonlinear} partial differential equations. \vf{They are still under development and broadly applied in various scientific topics such as groundwater flows \cite{Fabien2020high,etangsale:hal-03247309}, fluid dynamics \cite{kirk2019analysis}, solid mechanics \cite{SEVILLA201969}, wave propagation \cite{SANCHEZ2017951}, or magneto-hydro-dynamics \cite{Lee2019AnalysisOA}, to name but a few.} Indeed, they inherit attractive features from both (i) discontinuous Galerkin (DG) methods such as local conservation, $hp$-adaptivity and high-order polynomial approximation \cite{arnold2002unified} and (ii) standard conforming Galerkin (CG) methods such as the Schur complement strategy \cite{kirby2012cg}. One undeniable additional benefit of the HDG methods is their superconvergence property, obtained through the application of a local postprocessing technique on each element of the mesh \cite{nguyen2009implicit}. In the hybrid formalism, additional unknowns are introduced along the mesh skeleton corresponding to discrete trace approximations. Thanks to the specific localization of its additional degrees of freedom (dofs) \vf{and the discontinuous nature of approximation spaces}, interior variables can be \vf{locally} eliminated in favor of its Lagrange multipliers by only static condensation. \vf{The problem is then closed, and the algebraic linear system is assembled by imposing \textit{transmission} conditions throughout the mesh skeleton. This strategy is now well-established and -documented in the literature, and we refer the interested reader to the following works for a detailed description \cite{kirby2012cg,nguyen2009implicit,lehrenfeld2010hybrid} (see also, \cite[Remark 4.1]{etangsale:hal-03247309} for the description of the static condensation technique).} The resulting matrix system is significantly smaller and sparser than those associated with CG or DG methods for any given mesh and polynomial degree \cite{kirby2012cg}. Several HDG formulations have been derived in the literature and can be classified into two main categories. The first is based on a primal form of the continuous problem, such as the class of interior penalty (IP) methods \cite{FabienNM2019}, whereas the second relies on a dual (often called mixed) form, such as local discontinuous Galerkin (LDG) methods \cite{Cockburn09unified,nguyen2009implicit,DijouxA2M19}. In the latter formulation, the flux variable is introduced as an additional unknown of the problem.\newline
Our focus is on families of hybridizable interior penalty discontinuous Galerkin (H-IP) methods \cite{wells2011analysis}. They are hybridized counterparts of the well-known interior penalty DG (IPDG) methods \cite{arnold1982interior,Riviere08bookDG,dpe2011mathematical} and have been analyzed until quite recently by several authors \cite{FabienNM2019,kirk2019analysis}. Specifically, in our exposition, we considered the incomplete, non-symmetric and symmetric schemes denoted by H-IIP, H-NIP and H-SIP, respectively. The main difference between these schemes concerns the role of the \textit{symmetrization} term in the discrete bilinear form \cite{Riviere08bookDG}. Fabien \al recently
analyzed these schemes
using a stabilization function of the form $\mathcal{O}(1/h)$ for solving second-order elliptic problems \cite{FabienNM2019}. The authors conclude that H-IP methods inherit similar convergence properties to their IPDG equivalents. Notably, they theoretically establish (i) optimal energy error estimates, and because of the lack of symmetry of the associated discrete operator, (ii) only suboptimal $\LL{2}$-norm error estimates for H-IIP and H-NIP schemes. In addition, they numerically conclude that the $\LL{2}$‐orders of convergence of both non-symmetric variants are suboptimal for only even polynomial degrees and are optimal otherwise. Similar conclusions have also been suggested by Oikawa for second-order elliptic problems \cite{oikawa2017hdg}.\newline
To restore optimal $\LL{2}$-error estimates for the nonsymmetric IPDG method, Rivi\`ere \al suggest using a sort of superpenalty on the jumps \cite{Riviere98IPDG,Guzman09JSC}. In the present paper, we explore a similar idea in the general context of H-IP methods by using a variable penalty function of the form $\tau\eqbydef\mathcal{O}(1/h^{1+\delta})$, where $\delta\in\mathbb{R}$. Here, we analyze the direct impact of the parameter $\delta$ on \textit{a priori} error estimates in different norms. First, we propose a convergence analysis by investigating three key properties: (strong) consistency, discrete coercivity and boundedness. One remarkable feature of this strategy is the $h
^{\delta}$-dependency of the coercivity condition and the continuity (or boundedness) constant $\C{bnd}$, which consequently impacts the error estimates. Improved error estimates are then derived in the spirit of the second Strang lemma \cite{dpe2011mathematical}, and we first prove that the order of convergence in the natural energy-norm is linear, $\delta$-dependent, and optimal when $\delta\geq0$ for any scheme. Then, by using a duality argument, \ie the so-called Aubin--Nitsche technique, we also prove that the  optimal convergence is theoretically reached as soon as $\delta\geq 0$ for the H-SIP scheme only, and when $\delta\geq 2$ for both non-symmetric variants, \ie H-NIP and H-IIP schemes. Let us underline that we recover theoretical error estimates proposed in the literature for both the energy- and $\LL{2}$-norms if $\delta=0$.\newline
The rest of the material is organized as follows: Section \ref{S:2} describes the model problem, mesh notation and assumptions, and recalls some definitions and useful (trace) inequalities, while Section \ref{S:3} derives the discrete H-IP formulation and discusses its stability properties. In Section \ref{S:4}, optimal error estimates are provided for both the energy- and $\LL{2}$-norms by using a standard duality argument. Section \ref{S:5} concerns the numerical experiments that validate our theoretical results. We briefly end with some remarks and perspectives.

\section{Some preliminaries}
\label{S:2}
\subsection{The model problem}
Let $\domain$ be a bounded (polyhedron) domain in $\setR^d$ with Lipschitz boundary $\bnddom$ in spatial dimension $d\geq 2$. For clarity, we consider the anisotropic diffusion problem with homogeneous Dirichlet boundary conditions:
\begin{equation}
\label{model_problem}
- \nabla \cdot (\bkap \nabla u) = f\quad\textrm{in }\Omega\quad\textrm{and}\quad u=0\quad \textrm{on }\partial \Omega,
\end{equation}
where $\bkap\in[\lebesgue{\infty}{\domain}]^{d\times d}$ is a bounded, symmetric, uniformly positive-definite matrix-valued function and $f\in \lebesgue{2}{\Omega}$ is a forcing term. Thus, the weak formulation of problem \eqref{model_problem} is to find $u\in\hilbert{1}{\domain}{0}$ such that
\begin{equation}
    \label{weak_problem}
    \int_{\domain}\bkap\grd u\cdot\grd v\dx=\int_{\domain}fv\dx\quad\forall v\in\hilbert{1}{\domain}{0}.
\end{equation}
It is well known that under elliptic regularity assumptions, the variational problem \eqref{weak_problem} is well posed.

\subsection{Mesh notation and assumptions}

Let $h$ be a positive parameter; we assume without loss of generality that $h\leq 1$. We denote by $\{\Th\}_{h>0}$ a family of affine triangulations of the domain $\domain$, where $h$ stands for the largest diameter: $h_{\e{}}\eqbydef\diam{\e{}}$. We also assume that $\Th$ is \textit{quasi-uniform}, meaning that for all $\e{}\in\Th$, there exists $0<\rho_{0}\leq 1$ independent of $h$ such that $\rho_{0} h\leq h_{\e{}}\leq h$. Following our notation, the generic term \textit{interface} indicates a $(d-1)$-dimensional geometric object, \ie an edge, if $d=2$ and a face if $d=3$. Thus, we denote by $\Fhi$ and $\Fhb$ the set of interior and boundary interfaces, respectively. The set of all interfaces is called the mesh skeleton and is denoted by $\Fh\eqbydef\Fhi\cup\Fhb$. We denote by $\bnd{\Th}\eqbydef\{\cup\bnd{\e{}},\forall\e{}\in\Th\}$, the collection of interfaces of all mesh elements. Let $X$ be a mesh element or an interface; we then denote by $\abs{X}$ a positive $d$- or $(d-1)$-dimensional Lebesgue measure of $X$, respectively. Moreover, for any mesh element $\e{}\in\Th$, we denote by $\FE\eqbydef\{\f{}\in\Fh\,:\,\f{}\subset\bnd{\e{}} \}$ the set of interfaces composing the boundary of $\e{}$; we define $\eta_{\e{}}\eqbydef\card{\FE}$ and $\eta_0\eqbydef\max\limits_{\forall\e{}\in\Th}{(\eta_{\e{}})}$.

\subsection{Broken polynomial spaces}
For any polyhedral domain $D\subset\setR^d$ with $\partial D\subset\setR^{d-1}$, we denote by $\inerprod{\cdot}{\cdot}{D}$ (resp., $\dualprod{\cdot}{\cdot}{\partial D}$) the $\LL{2}$-inner product in $\lebesgue{2}{D}$ (resp., $\lebesgue{2}{\partial D}$) equipped with its natural norm $\norm{\cdot}{0}{D}$ (resp., $\norm{\cdot}{0}{\partial D}$). Let us now introduce some compact notation associated with the discrete $\LL{2}$-inner scalar product:
\begin{equation*}
\inerprod{\cdot}{\cdot}{\Th}\eqbydef\sum_{\e{}\in\Th}\inerprod{\cdot}{\cdot}{\e{}},\quad\dualprod{\cdot}{\cdot}{\partial\Th}\eqbydef\sum_{\e{}\in\Th}\dualprod{\cdot}{\cdot}{\bnd{\e{}}},
\end{equation*}
and we denote by $\norm{\cdot}{0}{\Th}$ and $\norm{\cdot}{0}{\bnd{\Th}}$ the corresponding norms. Similarly, we denote by $\hilbert{s}{D}{}$ the usual Hilbert space of index $s$ on $D$ equipped with its natural norm $\norm{\cdot}{s}{D}$ and seminorm $\seminorm{\cdot}{s}{D}$, respectively. If $s=0$, we then set $\hilbert{0}{D}{}=\lebesgue{2}{D}$. We denote by $\hilbert{s}{\Th}{}$ the usual broken Sobolev space and by $\grdh$ the broken gradient operator acting on $\hilbert{s}{\Th}{}$ with $s\geq 1$. We then assume an extended regularity requirement of the exact solution $u$ of the weak problem \eqref{weak_problem}, \ie $u\in\hilbert{s}{\domain}{0}\cap\hilbert{2}{\Th}{}$ with $s>3/2$. We also introduce the additional unknown $\tu\in\lebesgue{2}{\Fh}$ corresponding to the restriction of $u$ on the mesh skeleton; \ie $\tu\eqbydef u\vert_{\Fh}$. Let us now introduce the composite variable $\cu\eqbydef(u,\tu)$, which belongs to the continuous approximation space $\CV{}\eqbydef\hilbert{s}{\domain}{0}\cap\hilbert{2}{\Th}{}\times\lebesgue{2}{\Fh}$; \ie $\cu\in\CV{}$. As usual in HDG methods, we consider broken Sobolev spaces:
\begin{equation}
    \Pk{k}{\Th}\eqbydef\{\dv\in\lebesgue{2}{\Th}:\dv\vert_{\e{}}\in \Pk{k}{\e{}}, \forall\e{}\in\Th\},
\end{equation}
and similarly for $\Pk{k}{\Fh}$. Here, $\Pk{k}{X}$ denotes the space of polynomials of at least degree $k$ on $X$, where $X$ corresponds to a generic element of $\Th$ or $\Fh$, respectively. For H-IP discretization, two types of discrete variables are necessary to approximate the weak solution $u$ of problem \eqref{weak_problem}. First, the discrete variable $\du\in\Vh$ which is defined within each mesh element, and its trace $\dtu\in\TVh{}$, defined on the mesh skeleton with respect to the imposed homogeneous Dirichlet boundary conditions. To this aim, we set $\Vh\eqbydef\Pk{k}{\Th}$ and $\TVh{}\eqbydef\polyn{k}{\Fh}{0}$, where
\begin{align}
\polyn{k}{\Fh}{0}\eqbydef&\, \{ \dtv\in\polyn{k}{\Fh}{}\, : \, \dtv\vert_{\f{}} = 0,  \; \forall\f{}\in\Fhb  \}.
\end{align}
Throughout the manuscript, we use the following compact notations: Thus, let $\CVh{}\eqbydef\Vh\times\TVh{}$ be the composite approximation space and a generic element of $\CVh{}$ is denoted by $\cdv\eqbydef(\dv,\dtv)$. For all $\e{}\in\Th$ and $\f{}\in\FE$, we define the jump of $\cdv\in\CVh{}$ across $\f{}$ as $\tjump{\cdv}_{\e{},\f{}}\eqbydef(\dv\vert_{\f{}}-\dtv\vert_{\f{}})\mathbf{n}_{\e{},\f{}}$, where $\mathbf{n}_{\e{},\f{}}$ is the unit normal vector to $\f{}$ pointing out of $\e{}$. When confusion cannot arise, we omit the subscripts $\e{}$ and $\f{}$ from the definition, and we simply write $\tjump{\cdv}\eqbydef(\dv-\dtv)\mathbf{n}$. Finally, we introduce the space $\CVast\eqbydef\CV+\CVh{}$ to analyze the boundedness of the discrete bilinear form.

\subsection{Useful inequalities}
We recall here some useful inequalities that will be used extensively (see, \eg \cite{ciarlet1991basic,dpe2011mathematical,Riviere08bookDG}). For clarity, $\C{}$ denotes a generic constant that is independent of $h$, $h_{\e{}}$ and $\bkap$ in the rest of the manuscript. Owing to the shape regularity of $\Th$, we now introduce multiplicative trace inequalities. Let $\e{}\in\Th$ and $\f{}\in\FE$. For all $v\in\hilbert{2}{\e{}}{}$, there exists a constant $\C{M}>0$ independent of $h_{\e{}}$ and $v$ such that
\begin{subequations}
\begin{align}
    \label{multi_trace_ineq1}
    \norm{v}{0}{F}^2\leq\,&\C{M}(\norm{v}{0}{\e{}}\seminorm{v}{1}{\e{}}+h^{-1}_{\e{}}\norm{v}{0}{\e{}}^2),\\
    \label{multi_trace_ineq2}
    \norm{\grdh{v}}{0}{F}^2\leq\,&\C{M}(\seminorm{v}{1}{\e{}}\seminorm{v}{2}{\e{}}+h^{-1}_{\e{}}\seminorm{v}{1}{\e{}}^2).
\end{align}
\end{subequations}
Let us now remind the discrete and inverse trace inequalities, respectively. For all $\dv\in\Vh{}$, then the following holds
\begin{subequations}
\begin{align}
    \label{discret_trace_ineq}
    \norm{\dv}{0}{F}\leq\,&\C{tr}h^{-\ohalf}_{\e{}}\norm{\dv}{0}{\e{}},\\
    \label{inverse_trace_ineq}
    \norm{\grdh{\dv}}{0}{\e{}}\leq\,&\C{inv}h^{-1}_{\e{}}\norm{\dv}{0}{\e{}},
\end{align}
\end{subequations}
where $\C{tr}$ and $\C{inv}$ are positive constants independent of $h_{\e{}}$.
\begin{remark}
Following Rivi\`ere \cite{Riviere08bookDG} (see Section 2.1.3, p.24), one can obtain an exact expression of the constant $\C{tr}$ used in the discrete trace inequality \eqref{discret_trace_ineq} for a $d$-simplex mesh element:
\begin{equation}
    \label{const_ctr}
    \C{tr}\eqbydef\sqrt{\dfrac{(k+1)(k+d)}{d}},
\end{equation}
where $k$ denotes the polynomial degree of $\Vh{}$ and $d$ denotes the spatial dimension. This expression is particularly important in our analysis since it will be used later in the definition of the penalty parameter.
\end{remark}
We are now in a position to introduce the energy-norm used in the stability analysis and error estimations \cite{wells2011analysis,lehrenfeld2010hybrid}. For any given composite function $\cdv\in\CVh{}$, we consider the jump seminorm:
\begin{equation}
    \hdgseminorm{\cdv}{\gamma}^2\eqbydef\sum_{\e{}\in\Th}\sumFe\norm{\gamma^{\ohalf}_{\e{},\f{}}\tjump{\cdv}}{0}{\f{}}^2,
\end{equation}
where $\gamma_{\e{},\f{}}\geq 0$ is an arbitrary positive constant associated with $\f{}\in\FE$. The natural energy-norm equipping the discrete approximation space $\CVh{}$ is given by
\begin{equation}
    \label{nrj_norm1}
    \hdgnorm{\cdv}{\ast}^2\eqbydef\norm{\bkap^{\ohalf}\grdh{\dv}}{0}{\Th}^2+\hdgseminorm{\cdv}{\gamma}^2,
\end{equation}
which clearly depends on $\bkap$.

\section{Hybridizable interior penalty methods}
\label{S:3}
The discrete H-IP problem is to find $\cdu\in\CVh{}$ such that
\begin{equation}
\label{discrete_problem}
\dbilin{\mathcal{B}_h^{(\epsilon)}}{\cdu}{\cdv} = \dlin{l}{\cdv},\quad\forall\cdv\in\CVh{},
\end{equation}
where $\dlin{l}{\cdv}:=\inerprod{f}{\dv}{\Th}$. Here, the bilinear form $\mathcal{B}_h^{(\epsilon)}\,:\,\CVh{}\times\CVh{}\rightarrow\setR$ is given by
\begin{align}
\label{bilinear_form}
\dbilin{\mathcal{B}_h^{(\epsilon)}}{\cdu}{\cdv}&\eqbydef \inerprod{\bkap \grdh\du}{\grdh\dv}{\Th} +\dualprod{\tau\tjump{\cdu}}{\tjump{\cdv}}{\partial\Th} \nonumber\\
&-\dualprod{\bkap \grdh\du}{\tjump{\cdv}}{\partial\Th} -
 \epsilon \dualprod{\bkap \grdh\dv}{\tjump{\cdu}}{\partial\Th},
\end{align}
where $\epsilon\in\{0,\pm1\}$. The second, third and fourth terms on the right-hand side of \eqref{bilinear_form} are called the jump-penalty, consistency, and symmetry terms, respectively. The discrete bilinear operator $\mathcal{B}_h^{(\epsilon)}$ is symmetric iff $\epsilon=1$ and is nonsymmetric otherwise. We obtain the symmetric scheme (H-SIP) if $\epsilon=1$, the incomplete scheme (H-IIP) if $\epsilon=0$ and the nonsymmetric scheme (H-NIP) if $\epsilon=-1$. For all $\e{}\in\Th$ and $\f{}\in\FE$, the penalty term is chosen as follows:
\begin{equation}
    \label{penalty_term}
    \tau_{\e{},\f{}}\eqbydef\dfrac{\alpha_{0}\C{tr}^2\kappa_{\e{},\f{}}}{h_{\e{}}^{1+\delta}}\quad\textrm{with}\quad\delta\in\mathbb{R},
\end{equation}
where $\alpha_{0}$ is a user-dependent parameter, $\C{tr}$ is given by \eqref{const_ctr} and results from the discrete trace inequality \eqref{discret_trace_ineq}, 
and $\kappa_{\e{},\f{}}\eqbydef\vect{n}_{\e{},\f{}}\bkap_{\e{}}\vect{n}_{\e{},\f{}}$ denotes the normal diffusivity. 
\begin{remark}
For simplicity, we assume that $\bkap$ is approximated by piecewise constants on the mesh element $\Th$; \ie $\bkap_{\vert\e{}}\in\setR^{d\times d}$ for all $\e{}\in\Th$.
\end{remark}
\begin{lemma}[Consistency]
\label{lemma_consistency}
Let $\cu\in\CV$ be the compact notation of the exact solution of the problem \eqref{weak_problem}. For all $\cdv\in\CVh{}$, then the following holds:
\begin{equation}
\label{consistency}
\mathcal{B}_h^{(\cdot)} (\vect{u},\cdv) = l(\cdv).
\end{equation}
\end{lemma}
\begin{proof}
The regularity of $\cu\eqbydef(u,\tu)$ implies that its jump (in the HDG sense) is null on $\bnd{\Th{}}$, \ie for all $\e{}\in\Th{}$ and $\f{}\in\FE$ then $\tjump{\cu}\eqbydef 0$, since $u$ is a single-valued field on the mesh skeleton. Thus, by setting $\cdv\eqbydef(\dv,0)$, and integrating by parts on each element of the mesh, the bilinear form $\mathcal{B}_h^{(\epsilon)}$ yields
\begin{align}
\label{local_consist}
\dbilin{\mathcal{B}_h^{(\epsilon)}}{\vect{u}}{(\dv,0)}\eqbydef&\sum_{\e{}\in\Th}\inerprod{\dvgh(-\bkap\grdh u)}{\dv}{\e{}}=\sum_{\e{}\in\Th}\inerprod{f}{\dv}{\e{}}.
\end{align}
Considering now that $\cdv\eqbydef(0,\dtv)\in\CVh{}$ and $\dtv$ vanishes on the boundary skeleton $\Fhb$, we then obtain
\begin{equation}
    \label{trans_consist}
    \dbilin{\mathcal{B}_h^{(\epsilon)}}{\vect{u}}{(0,\dtv)}\eqbydef\dualprod{(\bkap\grdh u)\cdot\vect{n}}{\dtv}{\partial\Th}=0,
\end{equation}
which corresponds to the transmission conditions. The proof is then completed by summing \eqref{local_consist} and \eqref{trans_consist}.

\end{proof}
A straightforward consequence of Lemma \ref{lemma_consistency} is the Galerkin orthogonality.
\begin{proposition}[Galerkin orthogonality]
\label{proposition_galerkin_ortho}
Let $\cu\in\CV$ be the compact notation of the exact solution of the problem \eqref{weak_problem}, and $\cdu\in\CVh{}$, the solution of the discrete problem (\ref{discrete_problem}). Then,
\begin{equation}
    \label{galerkin_ortho}
    \dbilin{\mathcal{B}_h^{(\cdot)}}{\vect{u}-\cdu}{\cdv}=0\quad \forall\cdv\in\CVh{}.
\end{equation}
\end{proposition}
\begin{proof}
Subtracting \eqref{consistency} and \eqref{discrete_problem} yields the assertion.
\end{proof}

\subsection{Coercivity and well-posedness}

The next step is to prove discrete coercivity of  $\mathcal{B}_h^{(\cdot)}$ to ensure the well-posedness of \eqref{discrete_problem}. To this end, we first need to establish an upper bound of the consistency term using the jump seminorm $\hdgseminorm{\cdot}{\tau}$.
\begin{lemma}[Bound on consistency term]
\label{lemma_bound_consistency}
There exists a constant $\C{\delta}>0$ which is $h^{\delta}$-dependent such that
\begin{align}
    \abs{\dualprod{\bkap \grdh\dw}{\tjump{\cdv}}{\partial\Th}}\leq \C{\delta}^{\ohalf}\norm{\bkap^{\ohalf}\grdh\dw}{0}{\Th}\hdgseminorm{\cdv}{\tau},
\end{align}
for all $(\cdw,\cdv)\in\CVh{}\times\CVh{}$. Here $\C{\delta}\eqbydef\C{0}h^{\delta}$ and $\C{0}\eqbydef\C{}\eta_0/\alpha_0$ is a constant dependent of the element shape only.
\end{lemma}
\begin{proof}
The decomposition of the consistency term yields
\begin{align*}
    \dualprod{\bkap \grdh\dw}{\tjump{\cdv}}{\partial\Th}=\sumT\dualprod{\bkap\grdh\dw}{\tjump{\cdv}}{\bnd{\e{}}}.
\end{align*}
Applying the Cauchy--Schwarz inequality, using the definition of $\tau$ given in \eqref{penalty_term} and finally applying the discrete trace inequality \eqref{discret_trace_ineq}, we infer that
\begin{align*}
    \abs{\dualprod{\bkap\grdh\dw}{\tjump{\cdv}}{\bnd{\e{}}}}\leq&\norm{\bkap^{\ohalf}\grdh\dw}{0}{\bnd{\e{}}}\norm{\bkap^{\ohalf}\tjump{\cdv}}{0}{\bnd{\e{}}},\\
   \leq&\bigg[\dfrac{h^{1+\delta}_{\e{}}}{\alpha_{0}\C{tr}^2}\bigg]^{\ohalf}\norm{\bkap^{\ohalf}\grdh\dw}{0}{\bnd{\e{}}}\seminorm{\cdv}{\tau}{\bnd{\e{}}},\\
   \leq&\bigg[\dfrac{\eta_{\e{}}h^{\delta}_{\e{}}}{\alpha_{0}}\bigg]^{\ohalf}\norm{\bkap^{\ohalf}\grdh\dw}{0}{\e{}}\seminorm{\cdv}{\tau}{\bnd{\e{}}}.
\end{align*}
Considering now the quasi-uniformity requirement of the partition $\Th$ -- \ie for all $\e{}\in\Th$ and $\delta\in\mathbb{R}$ there exists $\C{}$ such that $h^{\delta}_{\e{}}\leq Ch^{\delta}$, we thus obtain
\begin{align*}
    \abs{\dualprod{\bkap\grdh\dw}{\tjump{\cdv}}{\bnd{\e{}}}}\leq\bigg[\dfrac{\C{}\eta_{0}h^{\delta}}{\alpha_{0}}\bigg]^{\ohalf}\norm{\bkap^{\ohalf}\grdh\dw}{0}{\e{}}\seminorm{\cdv}{\tau}{\bnd{\e{}}}.
\end{align*}
The proof is thus completed by summing over all mesh elements, and applying the Cauchy--Schwarz inequality.
\end{proof}
\begin{lemma}[Coercivity]
\label{lemma_coercivity}
Let us first introduce the (minimal) threshold value of the following form $\underline{\alpha}_{\epsilon,\delta}\eqbydef\C{\epsilon}h^{\delta}$ where $\C{\epsilon}\geq 0$ which is null iif $\epsilon=-1$. If the penalty parameter $\alpha_0$ in \eqref{penalty_term} is chosen large enough, \ie $\alpha_0>\underline{\alpha}_{\epsilon,\delta}$, then the discrete bilinear form $\mathcal{B}_h^{(\cdot)}$ is $\CVh{}$-coercive with respect to the energy-norm $\hdgnorm{\cdot}{\ast}$; \ie for all $\cdv\in\CVh{}$, then the following holds
\begin{equation}
    \label{coercivity}
    \dbilin{\mathcal{B}_h^{(\cdot)}}{\cdv}{\cdv}\geq \dfrac{1}{2}\hdgnorm{\cdv}{\ast}^2.
\end{equation}
\end{lemma}
\begin{proof}
Setting $\cdu=\cdv$ in \eqref{bilinear_form}, we thus obtain
\begin{align*}
\dbilin{\mathcal{B}_h^{(\epsilon)}}{\cdv}{\cdv} =  \norm{\bkap^{\ohalf}\grdh\dv}{0}{\Th}^2+\hdgseminorm{\cdv}{\tau}^2-
(1+\epsilon)\dualprod{\bkap \grdh\dv}{\tjump{\cdv}}{\partial\Th},
\end{align*}
proving immediately the coercivity of H-NIP scheme ($\epsilon=-1$) for any given value $\alpha_{0}>0$. Else, owing to Lemmata \ref{lemma_bound_consistency} and using Young's inequality, for any $0<\zeta<1$, there exists a constant $\C{\zeta}>0$ such that
\begin{align*}
\dbilin{\mathcal{B}_h^{(\epsilon)}}{\cdv}{\cdv} \geq &\bigg[1-\dfrac{\C{\delta}}{\zeta}\bigg]\norm{\bkap^{\ohalf}\grdh\dv}{0}{\Th}^2+(1-\zeta)\hdgseminorm{\cdv}{\tau}^2\\
\geq&\C{\zeta}\hdgnorm{\cdv}{\ast}^2,
\end{align*}
where $\C{\zeta}\eqbydef\min(1-\C{\delta}/\zeta,1-\zeta)$.
We now select $\alpha_{0}$ in the definition of $\C{\delta}$ such that $\C{\zeta}\eqbydef1-\zeta$; \ie $\C{\delta}<\zeta^2$ or equivalently by assuming that $\alpha_{0}>\zeta^{-2}\C{}\eta_0h^{\delta}$. The proof is thus completed by setting (arbitrary) $\zeta=1/2$.
\end{proof}

\begin{remark}
Note here the $h^{\delta}$-dependency of the coercivity condition of both H-SIP and H-IIP schemes. A straightforward consequence of the consistency and coercivity requirements via the Lax--Milgram Theorem is the well-posedness of the weak problem \eqref{discrete_problem}; \ie the existence and uniqueness of $\cdu\in\CVh{}$ are ensured.
\end{remark}

\subsection{Boundedness}

We now assume that the bilinear form $\mathcal{B}_h^{(\epsilon)}$ can be extended to $\CVast\times\CVast$, and we assert the boundedness of the product space. To this end, we introduce the enriched energy-norm on $\CVast$ denoted by $\tnorm{\cdot}$ (which is also a natural norm on $\CVh{}$) to bound the (normal) derivative terms \cite{lehrenfeld2010hybrid}.  For all $\cv\in\CVast$, then we set
\begin{equation}
    \label{nrj_norm2}
    \tnorm{\cv}^2\eqbydef\hdgnorm{\cv}{\ast}^2+\sumT h_{\e{}}\norm{\bkap^{\ohalf}\grdh v}{0}{\bnd{\e{}}}^2.
\end{equation}
\begin{lemma}[Equivalency of $\hdgnorm{\cdot}{\ast}$- and $\tnorm{\cdot}$-norms] \label{norm_equivalency}
The energy norms $\hdgnorm{\cdot}{\ast}$ and $\tnorm{\cdot}$ given by \eqref{nrj_norm1} and \eqref{nrj_norm2}, respectively, are uniformly equivalent on $\CVh{}$; \ie there exists a constant $\rho>0$ such that
\begin{equation}
    \label{norm_equivalence}
    \forall\cdv\in\CVh{},\quad\inv{\rho}\tnorm{\cdv}\leq \hdgnorm{\cdv}{\ast}\leq\tnorm{\cdv},
\end{equation}
where $\rho\eqbydef(1+\eta_{0}\C{tr}^2)^{\frac{1}{2}}$ depends only on the element shape.
\end{lemma}
\begin{proof}
Following the definition \eqref{nrj_norm2}, we first notice that $\hdgnorm{\cdv}{\ast}\leq\tnorm{\cdv}$. As $\cdv$ is piecewise polynomial, we now can easily bound the difference of both norms by using the discrete trace inequality \eqref{discret_trace_ineq},
\begin{equation*}
\tnorm{\cdv}^2 - \hdgnorm{\cdv}{\ast}^2 \leq \eta_0 \C{tr}^2  \norm{\bkap^{\ohalf}\grdh{\dv}}{0}{\Th}^2
\leq \eta_{0} \C{tr}^2 \hdgnorm{\cdv}{\ast}^2,
\end{equation*}
which yields the assertion.
\end{proof}

\begin{lemma}[Boundedness with $h^{\delta}$-dependency]
\label{lemma_boundedness}
There exists a constant $\C{bnd}>0$ which is $h^{\delta}$-dependent such that
\begin{equation}
\label{boundedness}
\forall(\cw,\cv)\in \CVast\times\CVast,\quad\dbilin{\mathcal{B}_h^{(\epsilon)}}{\cw}{\cv} \leq \C{bnd}\tnorm{\cw}\,\tnorm{\cv},
\end{equation}
where $\C{bnd}\eqbydef\max(2,\C{1}h^{\delta})$ and $\C{1}\eqbydef\inv{(\alpha_{0}\C{tr}^2)}$ is a constant independent of $h$.
\end{lemma}
\begin{proof}
The bilinear form \eqref{bilinear_form} can be decomposed as follows:
\begin{align}
\abs{\dbilin{\mathcal{B}_h^{(\epsilon)}}{\cw}{\cv}}\leq \abs{\T{1} + \T{2}} + \abs{\T{3}} + \abs{\T{4}},
\end{align}
where each terms are given below by
\begin{align*}
    \T{1}\eqbydef& \inerprod{\bkap^{\ohalf}\grdh w}{\bkap^{\ohalf}\grdh v}{\Th},\\
    \T{2}\eqbydef&\dualprod{\tau^{\ohalf}\tjump{\cw}}{\tau^{\ohalf}\tjump{\cv}}{\partial\Th}, \\ \T{3}\eqbydef&\dualprod{\bkap^{\ohalf}\grdh w}{\bkap^{\ohalf}\tjump{\cv}}{\partial\Th}.
\end{align*}
The last term $\T{4}$ is deduced from $\T{3}$ by permuting the role of $\cw$ and $\cv$, respectively. Thus, 
applying the Cauchy--Schwarz inequality, the first two terms can be bounded as follows:
\begin{subequations}
\begin{align*}
    \abs{\T{1}+\T{2}}\leq&\;[\norm{\bkap^{\ohalf}\grdh{w}}{0}{\Th}^2+\hdgseminorm{\cw}{\tau}^2]^{\ohalf}[\norm{\bkap^{\ohalf}\grdh{v}}{0}{\Th}^2+\hdgseminorm{\cv}{\tau}^2]^{\ohalf},\\
    \leq&\;\hdgnorm{\cw}{\ast}\hdgnorm{\cv}{\ast}.
\end{align*}
\end{subequations}
Proceeding as in the proof of Lemmata \ref{lemma_bound_consistency}, the third term can also be bounded as follows:
    \begin{align}
        \abs{\T{3}}\leq &\,\bigg[\C{1}h^{\delta}\sumT h_{\e{}}\norm{\bkap^{\ohalf}\grdh{w}}{0}{\bnd{\e{}}}^2\bigg]^{\ohalf}\hdgnorm{\cv}{\ast},
    \end{align}
where $\C{1}\eqbydef\inv{(\alpha_{0}\C{tr}^2)}$, and similarly for the fourth term $\abs{\T{4}}$. Collecting these estimates, and finally using the Cauchy--Schwarz inequality, we thus obtain
\begin{align*}
\abs{\dbilin{\mathcal{B}_h^{(\epsilon)}}{\cw}{\cv}}
\leq & \,\bigg[2\hdgnorm{\cw}{\ast}^2+\C{1}h^{\delta}\sumT h_{\e{}}\norm{\bkap^{\ohalf}\grdh{w}}{0}{\bnd{\e{}}}^2\bigg]^{\ohalf}\times\\ & \,\bigg[2\hdgnorm{\cv}{\ast}^2+\C{1}h^{\delta}\sumT h_{\e{}}\norm{\bkap^{\ohalf}\grdh{v}}{0}{\bnd{\e{}}}^2\bigg]^{\ohalf}, \\
\leq &\, \max(2,\C{1}h^{\delta})\tnorm{\cw}\tnorm{\cv},
\end{align*}
which yields the assertion.
\end{proof}

\begin{remark}
\label{remark1}
Let us emphasize here that $\C{bnd}\leq\C{}h^{r_{\delta}}$, where $r_{\delta}=\min{(0,\delta)}$ and $\C{}\eqbydef2\max(2,\C{1})$ is a positive constant independent of $h$.
\end{remark}

\section{A priori error analysis}
\label{S:4}
We now derive \textit{a priori} error estimates in both the discrete energy- and $\norm{\cdot}{0}{\Th}$-norms. The first ingredient of our error analysis is a bound on the quantity $\cu-\proj{\cu}$ in the energy-norms where $\proj{\cu}$ denotes a suitable \textit{continuous} interpolant of the compact solution $\cu$ of the problem \eqref{weak_problem}. Different authors have previously employed this trick for the error analysis of DG methods (see, \eg \cite{wells2011analysis} for a detailed description) since it offers several advantages. To this end, we recall standard interpolation estimates that will be used extensively in the rest of the document (see, \eg\cite{dpe2011mathematical,ciarlet1991basic}). Let us consider $\phi\in\hilbert{s}{\domain}{}$ with $s\geq 2$, and we denote by $\proji{\phi}$ its \textit{continuous} interpolant of degree $k$. 
Thus, the following estimates hold,
\begin{subequations}
    \label{interp_estimates}
    \begin{align}
        \label{interp_estimate1}
        \seminorm{\phi-\proji{\phi}}{q}{\Th}\leq\,\C{}h^{\mu-q}\seminorm{\phi}{\mu}{\Th},\quad\forall q\in\{0,\ldots,s-1\},\\
        \label{interp_estimate2}
        \bigg[\sumT h_{\e{}}^\alpha\norm{\grdh{(\phi-\proji{\phi})}}{0}{\bnd{\e{}}}^{2}\bigg]^{\ohalf}\leq\,\C{}h^{\mu+\frac{\alpha-3}{2}}\seminorm{\phi}{\mu}{\Th},
    \end{align}
\end{subequations}
where $\mu\eqbydef\min(k+1,s)$ and $k$ denotes the polynomial degree of the approximation space $\Vh{}$.
\begin{lemma}[Optimal error estimates]
\label{lemma_optimal_estimate}
Let $\cu\in\CV$ be the compact notation of the exact solution of the problem \eqref{weak_problem}. We denote by $\proj{\cu}\eqbydef(\proji{u},\projs{\tu})$ its continuous interpolant, 
where $\proji{u}\in\Vh{}\cap C^{0}(\bar{\Omega})$ and $\projs{\tu}\eqbydef\proji{u}\vert_{\Fh}$ which is contained in $\TVh{}$. Setting $\cerp{\cu}{\pi}\eqbydef\cu-\proj{\cu}$ then, the following holds
\begin{align}
    \label{optimal_estimate}
    \hdgnorm{\cerp{\cu}{\pi}}{\ast}\;\;\textrm{(or equiv.)}\;\;\tnorm{\cerp{\cu}{\pi}}\leq\C{\bkap}h^{\mu-1}\seminorm{u}{\mu}{\Th},
\end{align}
where $\mu\eqbydef\min(k+1,s)$ and $\C{\bkap}\eqbydef\C{}\norm{\bkap^{\ohalf}}{\infty}{\domain}$.
\end{lemma}
\begin{proof}
Setting $\cerp{\cu}{\pi}\eqbydef(\erp{u}{\pi},\terp{u}{\pi})$ where $\erp{u}{\pi}\eqbydef u-\proji{u}$ and $\terp{u}{\pi}\eqbydef \tu-\projs{\tu}$, and using the definition of the $\hdgnorm{\cdot}{\ast}$-norm \eqref{nrj_norm1} yields
\begin{align}
    \label{nrjnorm1-interpolant}
    \hdgnorm{\cerp{\cu}{\pi}}{\ast}^2 = \norm{\bkap^{\ohalf}\grdh{\erp{u}{\pi}}}{0}{\Th}^2+\hdgseminorm{\cerp{\cu}{\pi}}{\tau}^2.
\end{align}
The last term of \eqref{nrjnorm1-interpolant}, \ie the jump semi-norm of $\cerp{\cu}{\pi}$, is null according to the conformity of the (continuous) interpolant $\proj{\cu}$, \ie $\norm{\erp{u}{\pi}-\terp{u}{\pi}}{0}{\bnd{\e{}}}=0$ (see, \eg \cite[Lemma 5.5, p. 102]{wells2011analysis}). Thus, successively using the Cauchy--Schwarz inequality, and the interpolation estimate \eqref{interp_estimate1} yields 
\begin{align}
    \label{nrj1-estimate-int}
    \hdgnorm{\cerp{\cu}{\pi}}{\ast}^2\leq \norm{\bkap^{\ohalf}}{\infty}{\domain}^2\seminorm{\erp{u}{\pi}}{1}{\Th{}}^2
    \leq \C{}^2\norm{\bkap^{\ohalf}}{\infty}{\domain}^2h^{2\mu-2} \seminorm{u}{\mu}{\Th}^2.
\end{align}
The proof of the second estimate follows by the same arguments. Successively using  the definition of the continuity norm \eqref{nrj_norm2}, and the estimates \eqref{interp_estimate2} (with $\alpha=1$) and \eqref{nrj1-estimate-int}, we thus infer that
\begin{align*}
\tnorm{\cerp{\cu}{\pi}}^2&\overset{\eqref{nrj_norm2}}{\eqbydef}\hdgnorm{\cerp{\cu}{\pi}}{\ast}^2+\sumT h_{\e{}}\norm{\bkap^{\ohalf}\grdh{\erp{u}{\pi}}}{0}{\bnd{\e{}}}^2,\\
&\overset{\eqref{interp_estimates}}{\leq} \norm{\bkap^{\ohalf}}{\infty}{\domain}^2\sumT(\seminorm{ \erp{u}{\pi}}{1}{\e{}}^2+ h_{\e{}}\norm{\grdh{\erp{u}{\pi}}}{0}{\bnd{\e{}}}^2),\\
&\hspace{0.08cm}\leq \C{}^2\norm{\bkap^{\ohalf}}{\infty}{\domain}^2h^{2\mu-2} \seminorm{u}{\mu}{\Th}^2,
\end{align*}
which yields the assertion.
\end{proof}
\begin{remark}
We emphasize the optimality of the estimate \eqref{optimal_estimate} given in Lemma \ref{lemma_optimal_estimate}, independently of $\delta$. This is due to the continuous nature of the interpolant $\proj{\cu}$, which by conformity belongs to the kernel of stabilization terms whose contributions can be suboptimal. Let us specify that this assertion is no longer ensured by using discontinuous interpolants. This estimate is in agreement with the one established by Wells in \cite{wells2011analysis} for $\delta=0$.
\end{remark}

\subsection{Energy-norms error estimates}

We now derive an error estimation of the discrete composite variable $\cdu$ in the natural energy-norms.
\begin{theorem}[Energy-norm estimates]
\label{energy_norm_estimate}
Let $\cu\in\CV$ be the compact notation of the exact solution of the problem \eqref{weak_problem}. We denote by $\cdu\in\CVh{}$ the approximate solution of the discrete problem (\ref{discrete_problem}). Setting  $\cer{\cu}\eqbydef\cu-\cdu$ then, for any value of the parameter $\delta$, the following estimates hold:
\begin{align}
    \label{nrj1-estimate}
    \hdgnorm{\cer{\cu}}{\ast}\;\;\textrm{(or equiv.)}\;\;\tnorm{\cer{\cu}} &\leq \C{\bkap}h^{\mu+r_{\delta}-1}\seminorm{u}{\mu}{\Th},
\end{align}
where $\mu\eqbydef\min(k+1,s)$, $r_{\delta}\eqbydef\min(0,\delta)$, and $\C{\bkap}\eqbydef\C{}\norm{\bkap^{\ohalf}}{\infty}{\domain}$.
\end{theorem}
\begin{proof}
By using the triangle inequality for the definition of the stability energy-norm \eqref{nrj_norm1}, we easily infer that
\begin{equation}
\label{triangle_inequality}
\hdgnorm{\cu-\cdu}{\ast} \leq \hdgnorm{\cu-\proj{\cu}}{\ast} + \hdgnorm{\proj{\cu}-\cdu}{\ast}.
\end{equation}
Only an upper bound on the last term of \eqref{triangle_inequality} remains to be established. Successively using the coercivity, Galerkin orthogonality, and boundedness, we then deduce that
\begin{align*}
\frac{1}{2}\hdgnorm{\proj{\cu}-\cdu}{\ast}^2 &\overset{\eqref{coercivity}}{\leq} \dbilin{\mathcal{B}_h^{(\epsilon)}}{\proj{\cu}-\cdu}{\proj{\cu}-\cdu},\\
&\overset{\eqref{galerkin_ortho}}{=} \dbilin{\mathcal{B}_h^{(\epsilon)}}{\proj{\cu}-\cu}{\proj{\cu}-\cdu},\\ &\overset{\eqref{boundedness}}{\leq} \C{bnd}\tnorm{\cu-\proj{\cu}}\tnorm{\proj{\cu}-\cdu}.
\end{align*}
Finally, considering that $\proj{\cu}-\cdu\in\CVh{}$ and using Lemma \ref{norm_equivalency}, we obtain an upper bound of this term;
\begin{equation}
    \label{up-bound-proj}
    \hdgnorm{\proj{\cu}-\cdu}{\ast}\leq 2\rho\C{bnd}\tnorm{\cu-\proj{\cu}}.
\end{equation}
Inserting into \eqref{up-bound-proj} into \eqref{triangle_inequality}, we then infer that
\begin{align*}
\hdgnorm{\cu-\cdu}{\ast}\leq(1+2\rho\C{bnd})\tnorm{\cu-\proj{\cu}}.
\end{align*}
The proof of the second estimate (in the continuity-norm \eqref{nrj_norm2}) follows by the same arguments. Finally proceeding as in Remark \ref{remark1}, we can conclude that there exists a constant $\C{}>0$ such that $1+2\rho\C{bnd}\leq \C{}h^{r_{\delta}}$, that we combine with the optimal error estimate \eqref{optimal_estimate} given in Lemma \ref{lemma_optimal_estimate}, hence yielding to the assertion. 
\end{proof}
\begin{corollar}[Strong-regularity solutions]
\label{Corrolar_Strong}
Besides the hypotheses of Theorem \ref{energy_norm_estimate}, assume $u\in\hilbert{k+1}{\domain}{0}$. Then, we have the following estimate
\begin{equation}
\label{strong_regularity_estimate}
\hdgnorm{\vect{u}-\vect{\du}}{\ast} \leq \C{\bkap}h^{k+r_{\delta}}\seminorm{u}{k+1}{\Th}.
\end{equation}
where $\C{\bkap}\eqbydef\C{}\norm{\bkap^{\ohalf}}{\infty}{\domain}$.
\end{corollar}
\begin{proof}
(Evident)
\end{proof}
\begin{remark}
Following Di Pietro and Ern \cite[Theorem 4.53, p. 160]{dpe2011mathematical}, since $\C{}$ in Corollary \ref{Corrolar_Strong} is independent of $\bkap$, the discrete method is said to be robust with respect to diffusion heterogeneities (observing that the energy-norms depend on $\bkap$). The given estimate \eqref{strong_regularity_estimate} indicates that the order of convergence in the energy-norm is linear and $\delta$-dependent, \ie suboptimal if $\delta< 0$ and optimal otherwise. This is in agreement with the estimate given by Fabien \al in the particular case, $\delta=0$ \cite[Theorem 3.5, p. 8]{FabienNM2019}.
\end{remark}

\subsection{$\LL{2}$-norm error estimate}

Using a standard Aubin--Nitsche duality argument, we now derive an improved $\LL{2}$-error estimate of the H-IP method in terms of the parameter $\delta$. To this end, we define an auxiliary function $\psi$ as the solution of the adjoint problem:
\begin{equation*}
\label{adjoint_problem}
- \dvg (\bkap \grd{\psi}) = u-\du  \quad\textrm{in }\domain,\quad\textrm{and}\quad \psi=0\quad \textrm{on }\bnd{\domain}.
\end{equation*}
By assuming elliptic regularity, the following estimate holds:
\begin{equation}
\norm{\psi}{2}{\domain} \leq \C{\bkap}\norm{u-\du}{0}{\domain},
\end{equation}
where $\C{\bkap}$ depends on the shape regularity (\ie the convexity) of $\domain$ and the distribution of $\bkap$ inside it \cite{ern2009discontinuous}. The weak-adjoint problem is to find $\psi\in\hilbert{2}{\domain}{}\cap\hilbert{1}{\domain}{0}$ such that 
\begin{align}
\label{weak_adjoint_problem}
\inerprod{\bkap \grdh{\psi}}{\grdh{v}}{\Th} - \dualprod{\bkap \grdh{\psi}\cdot\normal}{v}{\bnd{\Th}}=\inerprod{u-\du}{v}{\Th},
\end{align}
for all $v\in\hilbert{1}{\domain}{0}$. Let us now introduce the composite error variable $\cer{\cu}\eqbydef\cu-\cdu=(\er{u},\ter{u})$ where $\er{u}\eqbydef u-\du$ and $\ter{u}\eqbydef \tu-\dtu$.
By setting now $v\eqbydef\er{u}$ in \eqref{weak_adjoint_problem}, we obtain
\begin{align}
\label{weak_form_adjoint}
\norm{\er{u}}{0}{\Th}^2=\inerprod{\bkap \grdh{\psi}}{\grdh{\er{u}}}{\Th} -\dualprod{\bkap \grdh{\psi}}{\er{u}\normal}{\bnd{\Th}}.
\end{align}
From the regularity of the variables $\tu$, $\dtu$ and $\psi$, we deduce that $\dualprod{\bkap \grdh{\psi}}{\ter{u}\normal}{\bnd{\Th}}=0$. By embedding this condition in \eqref{weak_form_adjoint}, we obtain an equivalent reformulation of the weak-adjoint problem in terms of the discrete bilinear operator $\mathcal{B}_h^{(\epsilon)}$:
\begin{align}
\label{L2_error_initial}
\norm{\er{u}}{0}{\Th}^2 &= \inerprod{\bkap \grd{\psi}}{\grd{\er{u}}}{\Th} - \dualprod{\bkap \grd{\psi}}{\tjump{\cer{\cu}}}{\bnd{\Th}},\nonumber\\
&=\dbilin{\mathcal{B}_h^{(\epsilon)}}{\cpsi}{\cer{\cu}},
\end{align}
where $\cpsi\eqbydef(\psi,\hat{\psi})$. Following the definition of the bilinear form $\mathcal{B}_h^{(\epsilon)}$ \eqref{bilinear_form} and using the Galerkin orthogonality $\dbilin{\mathcal{B}_h^{(\epsilon)}}{\cer{\cu}}{\proj{\cpsi}}=0$, since $\proj{\cpsi}\in\CVh{}$ (see Proposition \ref{proposition_galerkin_ortho}), we easily infer 
\begin{align}
\dbilin{\mathcal{B}_h^{(\epsilon)}}{\cpsi}{\cer{\cu}}&=\dbilin{\mathcal{B}_h^{(\epsilon)}}{\cer{\cu}}{\cerp{\cpsi}{\pi}}-(1-\epsilon)\dualprod{\bkap \grd{\psi}}{\tjump{\cer{u}}}{\bnd{\Th}},\nonumber\\
&\eqbydef\mathcal{T}_{1}-(1-\epsilon)\mathcal{T}_2,
\end{align}
where $\cerp{\cpsi}{\pi}\eqbydef\cpsi-\proj{\cpsi}$. We will now determine an upper bound of the quantity $\norm{\er{u}}{0}{\Th}^2$. Owing to Lemmas \ref{lemma_boundedness} and \ref{lemma_optimal_estimate} and using the regularity assumption $\psi\in\hilbert{2}{\domain}{}$, we can bound the first term $\mathcal{T}_1$:
\begin{align}
    \abs{\mathcal{T}_{1}} \leq \; \C{bnd}\tnorm{\cerp{\cpsi}{\pi}}\tnorm{\cer{\cu}}\leq \; \C{\bkap}\C{bnd}h\norm{\psi}{2}{\domain}\tnorm{\cer{\cu}}.
    \label{adjoint_T1}
\end{align}
Using the trace inequality $\norm{\grdh{\psi}}{0}{\bnd{\Th}}\leq \C{}h^{-\ohalf}\norm{\psi}{2}{\domain}$ \cite{ciarlet1991basic}, the second term $\mathcal{T}_2$ can be bounded as follows:
\begin{align}
    \abs{\mathcal{T}_2} \leq  \C{\bkap}h^{\frac{1+\delta}{2}}\norm{\grdh\psi}{0}{\bnd{\Th}}\hdgseminorm{\cer{\cu}}{\tau}\leq\,\C{\bkap}h^{\frac{\delta}{2}}\norm{\psi}{2}{\domain}\tnorm{\cer{\cu}}.
    \label{adjoint_T2}
\end{align}
Combining \eqref{adjoint_T1} and \eqref{adjoint_T2}, we obtain the estimate
\begin{equation}
\norm{u-\du}{0}{\Th} \leq \C{\bkap}(\C{bnd}h+(1-\epsilon) h^{\frac{\delta}{2}})\tnorm{\cer{\cu}},
\label{AN_estimates}
\end{equation}
and we can assert the theorem below.
\begin{theorem}[$\LL{2}$-norm error estimate]
\label{L2_norm_estimate}
Besides the hypotheses of Theorem \ref{energy_norm_estimate} then, we have the estimate
\begin{align}
    \norm{u-\du}{0}{\Th} & \leq \C{\bkap} h^{\mu+s^{(\epsilon)}_{\delta}} \seminorm{u}{\mu}{\Th},
    \label{L2_norm_estimates}
\end{align}
where the constant $\C{\bkap}$ depends on the shape regularity of $\domain$ and the distribution of $\bkap$ inside it, and $s^{(\epsilon)}_{\delta}$ is only dependent on $\epsilon$ and $\delta$ and is given by
\begin{align}
    \label{order_L2}
    s^{(\epsilon)}_{\delta}\eqbydef\begin{cases}\min(0,2\delta)\equiv 2r_{\delta} & \textrm{if}\quad\epsilon=1,\\
    \min(0,\delta/2-1) & \textrm{if}\quad\epsilon\neq 1\quad\textrm{and}\quad\delta\geq 0,\\
    \min(2\delta,3\delta/2-1) & \textrm{if}\quad\epsilon\neq 1\quad\textrm{and}\quad\delta< 0.
    \end{cases}
\end{align}
\end{theorem}
\begin{proof}
The estimate \eqref{L2_norm_estimates} using \eqref{order_L2} follows after some algebraic manipulations from the previous equation \eqref{AN_estimates}, the definition of $\C{bnd}$ given in Lemma \ref{lemma_boundedness} and the error estimate (in the $\tnorm{\cdot}$-norm \eqref{nrj_norm2}) given in Theorem \ref{energy_norm_estimate}. 
\end{proof}
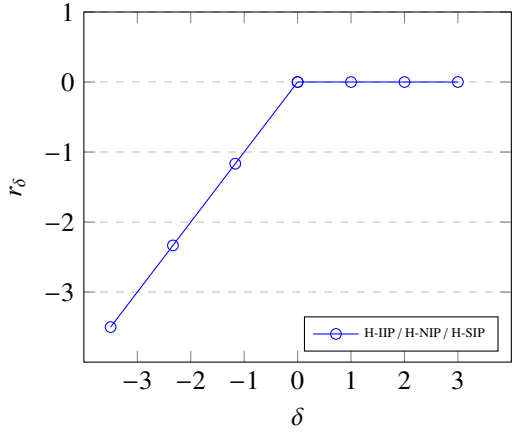
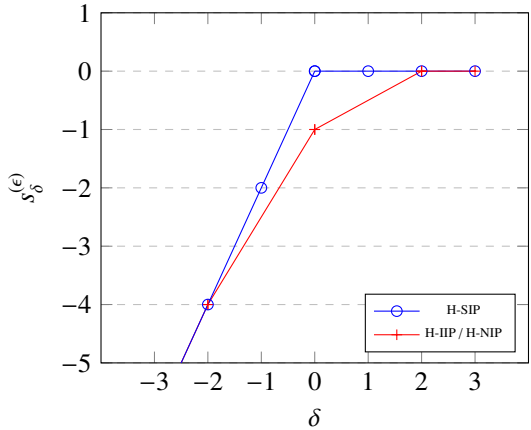
\begin{figure}[!ht]
\centering
\subfloat[]{\begin{tikzpicture}
\begin{axis}[
    xlabel={$\delta$},
    ylabel={$r_{\delta}$},
    xmin=-4, xmax=4,
    ymin=-4, ymax=1,
    xtick={-3,-2,-1,0,1,2,3},
    ytick={-3,-2,-1,0,1},
    legend pos=south east,
    ymajorgrids=true,
    grid style=dashed,
    scale=0.95]
\addplot[domain=0:3,mark=o, mark size=2,samples=4,color=blue]{0};
\addplot[domain=-3.5:0,mark=o, mark size=2,samples=4,color=blue]{x};
\addlegendentry{\tiny{H-IIP / H-NIP / H-SIP}}
\end{axis}
\end{tikzpicture}
}\\
\subfloat[]{\begin{tikzpicture}
\begin{axis}[
    xlabel={$\delta$},
    ylabel={$s^{(\epsilon)}_{\delta}$},
    xmin=-4, xmax=4,
    ymin=-5, ymax=1,
    xtick={-3,-2,-1,0,1,2,3},
    ytick={-5,-4,-3,-2,-1,0,1},
    legend pos=south east,
    ymajorgrids=true,
    grid style=dashed,
    scale=0.95]
    \addplot[domain=0:3,mark=o, mark size=2,samples=4, color=blue]{0};
    \addlegendentry{\tiny{H-SIP}}
    \addlegendentry{\tiny{H-IIP / H-NIP}}
    \addplot[domain=2:3,mark=+, mark size=2,samples=2, color=red]{0};
    \addplot[domain=0:2,mark=+, mark size=2,samples=2,color=red]{x/2-1};
    \addplot[domain=-2:0,mark=+, mark size=2,samples=2,color=red]{3*x/2-1};
    \addplot[domain=-3:-2,mark=+, mark size=2,samples=2,color=red]{2*x};
    \addplot[domain=-3:0,mark=o, mark size=2,samples=4,color=blue]{2*x}; 
\end{axis}
\end{tikzpicture}
}
\caption{Representation of the quantities $r_{\delta}$ and $s^{(\epsilon)}_{\delta}$ vs. $\delta$ given in Theorems \ref{energy_norm_estimate} and \ref{L2_norm_estimate}, respectively.}
\end{figure}
Let us emphasize that H-IP methods inherit similar asymptotic behaviors than their standard IPDG counterparts. Due to the lack of symmetry of both H-IIP and H-NIP schemes, the a priori error estimates in $\LL{2}$ are optimal only if $\delta\geq 2$.
We also point out that the estimate given in Theorem \ref{L2_norm_estimate} is in agreement with previous results established by different authors in the literature in the specific case $\delta=0$ (see, \eg\cite[Theorem 3.6, p. 9]{FabienNM2019} and \cite[Lemma 5.5, p. 103]{wells2011analysis} for the H-SIP method).

\begin{remark}
The authors are certain that these estimates given in Theorems \ref{energy_norm_estimate} and \ref{L2_norm_estimate} have already been established in the literature, but we have not been able to find them.  
\end{remark}

\section{Numerical experiments}
\label{S:5}

In the previous sections, we built families of hybridizable interior penalty methods based on an adaptive definition of the penalty parameter that depends on several coefficients. This section highlights
the benefit these methods provide
in the approximation of diffusion problems with anisotropic and/or discontinuous coefficients and in the validation of a priori error estimates. 
All numerical experiments are performed using the high-performance finite element library \textsc{NGSolve} \cite{Schoberl2014c++}. Then, the physical domain is taken to be a unit square---\ie $\domain\eqbydef[0,1]^2\subset\mathbb{R}^2$---and the right-hand-side $f$ is chosen such that the given exact solution $u$ respecting the homogeneous boundary conditions is verified. We use a sequence of subdivisions $\Th$, where regular triangles or squares form each partition (see, \eg Figure \ref{mesh_representation}). Standard $h$- and $k$-refinement strategies are used to compute the numerical errors and estimated convergence rates (ECRs). To pursue our quantitative analysis, we first measure the impact of the parameter $\delta$ on the a posteriori error estimates. Second, we point out the crucial role of the factor $\kappa_{n}$ arising in \eqref{penalty_term} for the robustness of the H-IP methods when the medium becomes highly anisotropic and/or discontinuous. Finally, we complete our experiments by pointing out some unexpected benefits of the value of $\alpha_{0}$ for the ECRs of the H-SIP scheme.
\begin{figure}[!ht]
\centering
\subfloat[\label{mesh1}]{\includegraphics[scale=0.10,trim = 500 130 500 130, clip=true]{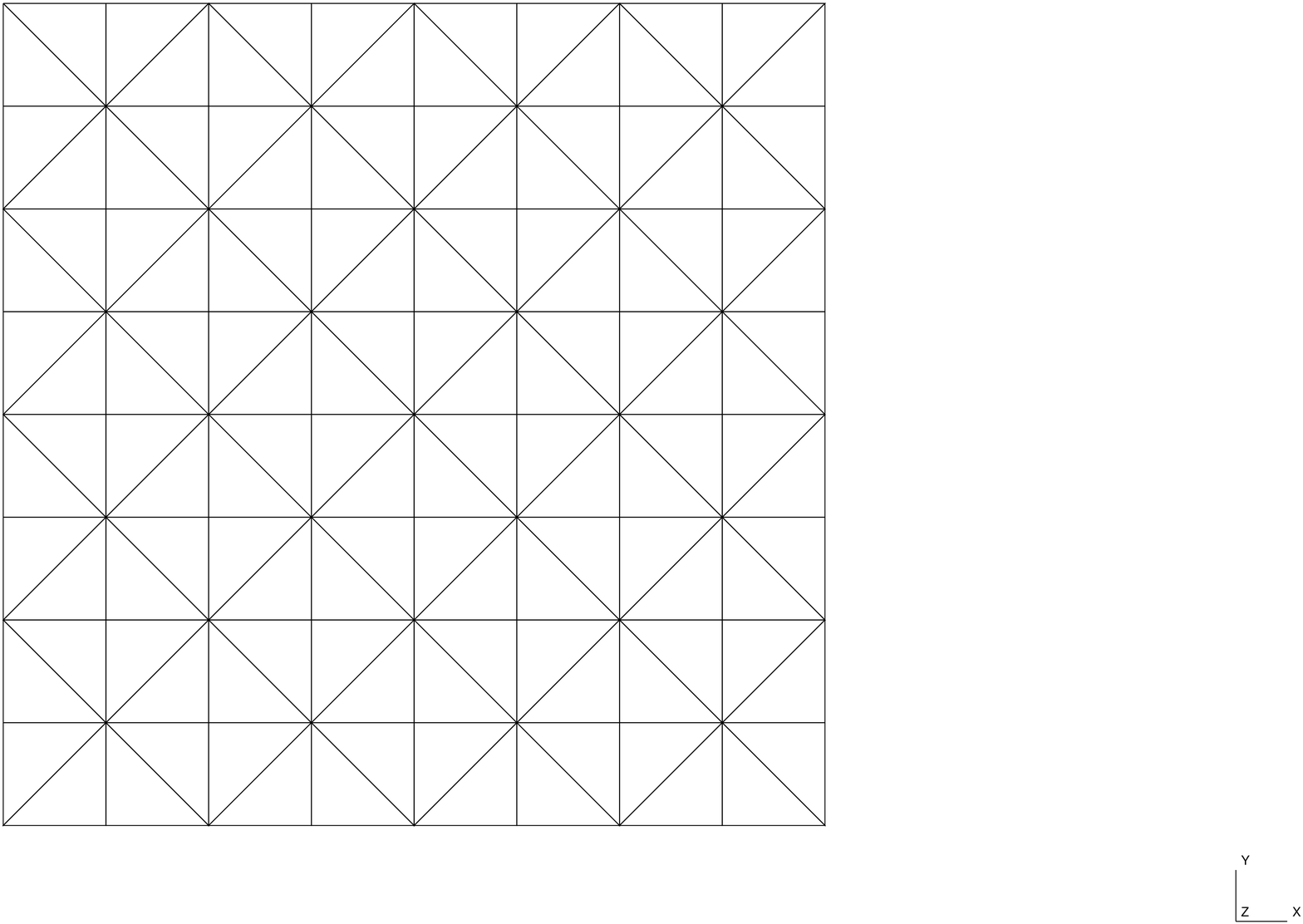}} \quad
\subfloat[\label{mesh2}]{\includegraphics[scale=0.10,trim = 500 130 500 130, clip=true]{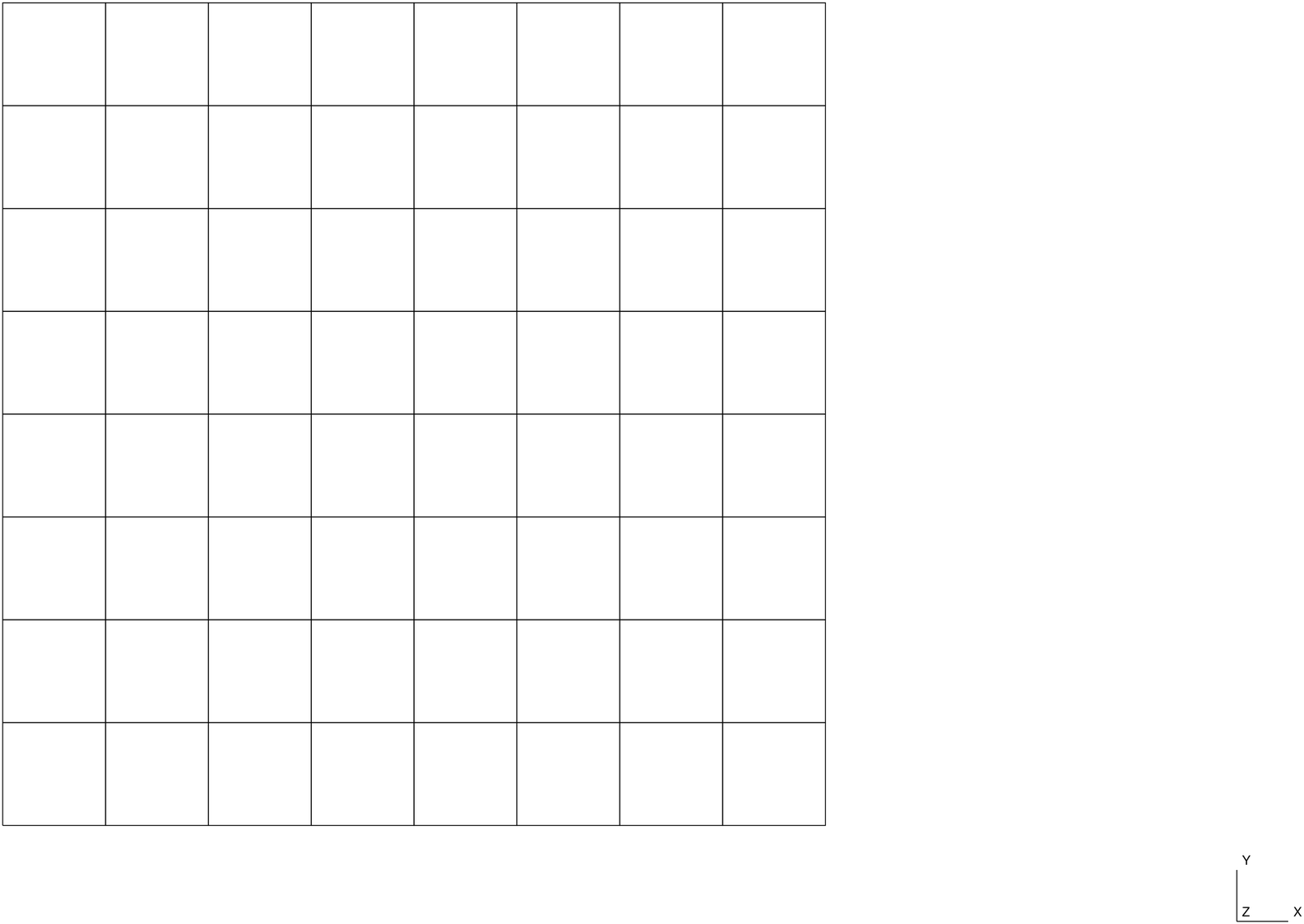}}
\caption{Uniform triangular (a) and square (b) meshes with $h=1/8$, respectively.}
\label{mesh_representation}
\end{figure}

\subsection{Test A: Influence of the parameter $\delta$}

We consider the following test case, which was previously proposed in Fabien \al \cite{FabienNM2019}: the diffusion tensor is homogeneous and isotropic---$\bkap \eqbydef \bm{I_2}$  (identity matrix)---and the exact smooth solution is given by $u(x,y)=xy(1-x)(1-y)\exp(-x
^2-y^2)$. Then, for all $\e{}\in\Th$ and for all $\f{}\in\FE$, we assume that the penalty parameter has the following simplified form:
\begin{equation}
\tau_{\e{},\f{}} \eqbydef 
\dfrac{\tau_0}{h_{\e{}}^{1+\delta}},
\end{equation}
where $\tau_0>0$ is a positive constant chosen to be large enough in accordance with Lemma \ref{lemma_coercivity}. The objective here is to measure the impact of the parameter $\delta$ on the ECRs in both the $\LL{2}$- and energy-norms. A history of convergence is shown in Figures \ref{convergence1} ($\tnorm{\cdot}$-norm) and \ref{convergence2} ($\norm{\cdot}{0}{\Th}$-norm) for uniform triangular meshes and for polynomial degrees $k\in\{1,\cdots,3\}$. 
As expected,
these observations are in agreement
with theoretical estimates and underline that the stabilization parameter $\delta$ influences the convergence rate. In particular, we recover some well-known estimates if $\delta=0$. First, we notice that the convergence of the H-IP method in the energy-norm is linearly $\delta$-dependent if $\delta\leq 0$ and optimal if $\delta\geq 0$, which is in accordance with Lemma \ref{energy_norm_estimate} (see \figref{convergence1}). A brief analysis of the convergence in the $\LL{2}$-norm indicates that both the H-IIP and H-NIP schemes behave differently from the H-SIP scheme. Nonsymmetric variants are strongly influenced by the polynomial parity of $k$ and by the penalty parameter $\delta$. We observe that the convergence rate increases linearly and optimally if $\delta\geq 0$ for odd $k$ and $\delta\geq 2$ for even $k$. In this last case, let us point out that the optimal convergence is nearly reached once $\delta\geq 1$. As expected, the symmetric scheme converges optimally when $\delta\geq 0$. These results agree with the theoretical results established in Theorem \ref{L2_norm_estimate}.

\begin{figure}[!ht]
\centering
\subfloat[]{\begin{tikzpicture}[scale=0.61]
\begin{loglogaxis}[ylabel={$\tnorm{u-u_h}$}, legend cell align=left, legend pos = south east, font=\small]

\addplot[color=Sepia,mark=o] coordinates { (1/16,6.0e-03) (1/32,3.0e-03) (1/64,1.5e-03) (1/128,7.4e-04)};
		
\addplot[color=BrickRed,mark=triangle] coordinates { (1/16,1.5e-03) (1/32,5.2e-04) (1/64,1.8e-04) (1/128,6.3e-05) };

\addplot[color=ForestGreen,mark=square] coordinates { (1/16,3.3e-04) (1/32,8.3e-05) (1/64,2.1e-05) (1/128,5.2e-06) };

\draw[color=black,dashed] (axis cs:1/64,1.5e-03) |- (axis cs:1/128,7.4e-04) node[near start,right]{$1.0$};

\draw[color=black,dashed] (axis cs:1/64,1.8e-04) |- (axis cs:1/128,6.3e-05) node[near start,right]{$1.5$};

\draw[color=black,dashed] (axis cs:1/64,2.1e-05) |- (axis cs:1/128,5.2e-06) node[near start,right]{$2.0$};

\legend{$\delta=-1$, $\delta=-1/2$, $\delta=0$, $\delta=1/2$}

\end{loglogaxis}
\end{tikzpicture} }
\subfloat[]{\begin{tikzpicture}[scale=0.61]
\begin{loglogaxis}[legend cell align=left, legend pos = south east, font=\small]

\addplot[color=Sepia,mark=o] coordinates {(1/8,7.3e-04) (1/16,1.6e-04) (1/32,3.8e-05) (1/64,9.1e-06) };
		
\addplot[color=BrickRed,mark=triangle] coordinates {(1/8,2.9e-04) (1/16,4.5e-05) (1/32,7.2e-06) (1/64,1.2e-06) };

\addplot[color=ForestGreen,mark=square] coordinates {(1/8,1.4e-04) (1/16,1.7e-05) (1/32,2.0e-06) (1/64,2.5e-07) };	

\draw[color=black,dashed] (axis cs:1/32,3.8e-05) |- (axis cs:1/64,9.1e-06) node[near start,right]{$2.0$};


\draw[color=black,dashed] (axis cs:1/32,2.0e-06) |- (axis cs:1/64,2.5e-07) node[near start,right]{$3.0$};

\legend{$\delta=-1$, $\delta=-1/2$, $\delta=0$}

\end{loglogaxis}
\end{tikzpicture} }\\
\subfloat[]{\begin{tikzpicture}[scale=0.61]
\begin{loglogaxis}[ylabel={$\tnorm{u-u_h}$}, legend cell align=left, legend pos = south east, font=\small]

\addplot[color=Sepia,mark=o] coordinates { (1/16,8.5e-04) (1/32,2.2e-04) (1/64,5.6e-05) (1/128,1.4e-05)};
		
\addplot[color=BrickRed,mark=triangle] coordinates { (1/16,7.7e-04) (1/32,2.0e-04) (1/64,5.1e-05) (1/128,1.3e-05)};

\addplot[color=ForestGreen,mark=square] coordinates { (1/16,5.2e-04) (1/32,1.3e-04) (1/64,3.2e-05) (1/128,8.1e-06)};

\draw[color=black,dashed] (axis cs:1/64,3.2e-05) |- (axis cs:1/128,8.1e-06) node[near start,right]{$2.0$};

\legend{$\delta=-1$, $\delta=-1/2$, $\delta=0$}

\end{loglogaxis}
\end{tikzpicture} }
\subfloat[]{\begin{tikzpicture}[scale=0.61]
\begin{loglogaxis}[legend cell align=left, legend pos = south east, font=\small]

\addplot[color=Sepia,mark=o] coordinates {(1/8,2.4e-04) (1/16,3.3e-05) (1/32,4.4e-06) (1/64,5.6e-07) };
		
\addplot[color=BrickRed,mark=triangle] coordinates {(1/8,1.8e-04) (1/16,2.5e-05) (1/32,3.4e-06) (1/64,4.2e-07) };

\addplot[color=ForestGreen,mark=square] coordinates {(1/8,1.2e-04) (1/16,1.6e-05) (1/32,1.9e-06) (1/64,2.4e-07) };	



\draw[color=black,dashed] (axis cs:1/32,2.0e-06) |- (axis cs:1/64,2.5e-07) node[near start,right]{$3.0$};

\legend{$\delta=-1$, $\delta=-1/2$, $\delta=0$}

\end{loglogaxis}
\end{tikzpicture} }\\
\subfloat[]{\begin{tikzpicture}[scale=0.61]
\begin{loglogaxis}[ylabel={$\tnorm{u-u_h}$}, legend cell align=left, legend pos = south east, font=\small]

		
\addplot[color=Sepia,mark=o] coordinates { (1/16,7.6e-04) (1/32,1.9e-04) (1/64,4.8e-05) (1/128,1.2e-05)};
		
\addplot[color=BrickRed,mark=triangle] coordinates { (1/16,7.7e-04) (1/32,2.0e-04) (1/64,5.0e-05) (1/128,1.2e-05)};

\addplot[color=ForestGreen,mark=square] coordinates { (1/16,4.5e-04) (1/32,1.1e-04) (1/64,2.8e-05) (1/128,6.9e-06)};

\draw[color=black,dashed] (axis cs:1/64,2.8e-05) |- (axis cs:1/128,6.9e-06) node[near start,right]{$2.0$};

\legend{$\delta=-1$, $\delta=-1/2$, $\delta=0$}

\end{loglogaxis}
\end{tikzpicture} }
\subfloat[]{\begin{tikzpicture}[scale=0.61]
\begin{loglogaxis}[legend cell align=left, legend pos = south east, font=\small]


\addplot[color=Sepia,mark=o] coordinates {(1/8,8.2e-04) (1/16,9e-05) (1/32,1.0e-05) (1/64,1.0e-06) };
		
\addplot[color=BrickRed,mark=triangle] coordinates {(1/8,8.0e-04) (1/16,8.9e-05) (1/32,9.9e-06) (1/64,1.1e-06) };

\addplot[color=ForestGreen,mark=square] coordinates {(1/8,2.2e-04) (1/16,2.8e-05) (1/32,3.5e-06) (1/64,4.3e-07) };	



\draw[color=black,dashed] (axis cs:1/32,3.5e-06) |- (axis cs:1/64,4.3e-07) node[near start,right]{$3.0$};

\legend{$\delta=-1$, $\delta=-1/2$, $\delta=0$}

\end{loglogaxis}
\end{tikzpicture} }
\caption{Test A: from the top to the bottom: history of convergence in the $\tnorm{\cdot}$-norm (vs. $h$) of the H-IIP (a-b), H-NIP (c-d) and H-SIP (e-f) schemes, respectively, on uniform triangular meshes with $-1\leq\delta\leq 0$. In the left images (a-c-e), $k=2$, and in the right images (b-d-f), $k=3$.}
\label{convergence1}
\end{figure}


\begin{figure}[!ht]
\centering
\subfloat[]{\begin{tikzpicture}[scale=0.61]
\begin{loglogaxis}[ylabel={$\norm{u-u_h}{0}{\Th}$},legend cell align=left, legend pos = south east, font=\tiny]
		
\addplot[color=Sepia,mark=o] coordinates {(1/16,1.9e-03) (1/32,9.5e-04) (1/64,4.8e-04) (1/128,2.4e-04) };
		
\addplot[color=BrickRed,mark=triangle] coordinates {(1/16,1.0e-04) (1/32,2.6e-05) (1/64,6.4e-06) (1/128,1.6e-06) };

\addplot[color=ForestGreen,mark=square] coordinates {(1/16,8.1e-06) (1/32,1.2e-06) (1/64,1.7e-07) (1/128,2.2e-08) };

\addplot[color=MidnightBlue,mark=diamond] coordinates {(1/16,3.4e-06) (1/32,4.3e-07) (1/64,5.4e-08) (1/128,6.8e-09) };

\draw[color=black,dashed] (axis cs:1/64,4.5e-04) |- (axis cs:1/128,2.2e-04) node[near start,right]{$1.0$};

\draw[color=black,dashed] (axis cs:1/64,6.4e-06) |- (axis cs:1/128,1.6e-06) node[near start,right]{$2.0$};

\draw[color=black,dashed] (axis cs:1/64,5.4e-08) |- (axis cs:1/128,6.8e-09) node[near start,right]{$3.0$};

\legend{$\delta=-1$, $\delta=0$, $\delta=1$, $\delta=2$}

\end{loglogaxis}
\end{tikzpicture} }
\subfloat[]{\begin{tikzpicture}[scale=0.61]
\begin{loglogaxis}[legend cell align=left, legend pos = south east, font=\tiny]

\addplot[color=Sepia,mark=o] coordinates {(1/8,7.1e-05) (1/16,8.9e-06) (1/32,1.1e-06) (1/64,1.4e-07) };
		
\addplot[color=BrickRed,mark=triangle] coordinates {(1/8,9.2e-06) (1/16,5.8e-07) (1/32,3.7e-08) (1/64,2.3e-09) };

\addplot[color=ForestGreen,mark=square] coordinates {(1/8,9.8e-07) (1/16,6.5e-08) (1/32,3.4e-09) (1/64,2.0e-10) };	

\addplot[color=MidnightBlue,mark=diamond] coordinates { (1/8,8.4e-07) (1/16,5.3e-08) (1/32,3.3e-09) (1/64,2.1e-10) };

\draw[color=black,dashed] (axis cs:1/32,1.1e-06) |- (axis cs:1/64,1.4e-07) node[near start,right]{$3.0$};

\draw[color=black,dashed] (axis cs:1/32,3.3e-09) |- (axis cs:1/64,2.1e-10) node[near start,right]{$4.0$};

\legend{$\delta=-1$, $\delta=0$, $\delta=1$, $\delta=2$}

\end{loglogaxis}
\end{tikzpicture} }\\
\subfloat[]{\begin{tikzpicture}[scale=0.61]
\begin{loglogaxis}[ylabel={$\norm{u-u_h}{0}{\Th}$},legend cell align=left, legend pos = south east, font=\tiny]

\addplot[color=Sepia,mark=o] coordinates {(1/16,1.3e-04) (1/32,3.5e-05) (1/64,9.0e-06) (1/128,2.3e-06) };
		
\addplot[color=BrickRed,mark=triangle] coordinates {(1/16,5.7e-05) (1/32,1.4e-05) (1/64,3.6e-06) (1/128,8.9e-07) };

\addplot[color=ForestGreen,mark=square] coordinates {(1/16,1.2e-05) (1/32,2.0e-06) (1/64,3.0e-07) (1/128,4.2e-08) };

\addplot[color=MidnightBlue,mark=diamond] coordinates {(1/16,3.6e-06) (1/32,4.4e-07) (1/64,5.5e-08) (1/128,6.8e-09) };

\draw[color=black,dashed] (axis cs:1/64,3.6e-06) |- (axis cs:1/128,8.9e-07) node[near start,right]{$2.0$};


\draw[color=black,dashed] (axis cs:1/64,5.5e-08) |- (axis cs:1/128,6.8e-09) node[near start,right]{$3.0$};

\legend{$\delta=-1$, $\delta=0$, $\delta=1$, $\delta=2$}

\end{loglogaxis}
\end{tikzpicture} }
\subfloat[]{\begin{tikzpicture}[scale=0.61]
\begin{loglogaxis}[legend cell align=left, legend pos = south east, font=\tiny]

\addplot[color=Sepia,mark=o] coordinates {(1/8,1.1e-05) (1/16,8.2e-07) (1/32,5.5e-08) (1/64,3.5e-09) };
		
\addplot[color=BrickRed,mark=triangle] coordinates {(1/8,5.2e-06) (1/16,3.3e-07) (1/32,2.1e-08) (1/64,1.3e-09) };

\addplot[color=ForestGreen,mark=square] coordinates {(1/8,1.2e-06) (1/16,6.8e-08) (1/32,4.e-09) (1/64,2.4e-10) };	

\addplot[color=MidnightBlue,mark=diamond] coordinates { (1/8,9.9e-07) (1/16,5.6e-08) (1/32,3.4e-09) (1/64,2.1e-10) };


\draw[color=black,dashed] (axis cs:1/32,3.3e-09) |- (axis cs:1/64,2.1e-10) node[near start,right]{$4.0$};

\legend{$\delta=-1$, $\delta=0$, $\delta=1$, $\delta=2$}

\end{loglogaxis}
\end{tikzpicture} }\\
\subfloat[]{\begin{tikzpicture}[scale=0.61]
\begin{loglogaxis}[ylabel={$\norm{u-u_h}{0}{\Th}$},legend cell align=left, legend pos = south east, font=\tiny]

		
\addplot[color=Sepia,mark=o] coordinates { (1/16,1.2e-05) (1/32,1.5e-06) (1/64,1.9e-07) (1/128,2.4e-08)};
		
\addplot[color=BrickRed,mark=triangle] coordinates { (1/16,4.3e-06) (1/32,5.3e-07) (1/64,6.7e-08) (1/128,8.3e-09) };




\draw[color=black,dashed] (axis cs:1/64,6.7e-08) |- (axis cs:1/128,8.3e-09) node[near start,right]{$3.0$};

\legend{$\delta=-1$, $\delta=0$}

\end{loglogaxis}
\end{tikzpicture} }
\subfloat[]{\begin{tikzpicture}[scale=0.61]
\begin{loglogaxis}[legend cell align=left, legend pos = south east, font=\tiny]


\addplot[color=Sepia,mark=o] coordinates {(1/8,3.5e-05) (1/16,2.1e-06) (1/32,1.3e-07) (1/64,8.1e-09) };
		
\addplot[color=BrickRed,mark=triangle] coordinates {(1/8,3.4e-06) (1/16,2.1e-07) (1/32,1.3e-08) (1/64,8.3e-10) };




\draw[color=black,dashed] (axis cs:1/64,8.3e-10) -| (axis cs:1/32,1.3e-08) node[near end,right]{$4.0$} node[near start,below]{$1$};

\legend{$\delta=-1$, $\delta=0$, $\delta=1$, $\delta=2$}

\end{loglogaxis}
\end{tikzpicture} }
\caption{Test A: from the top to the bottom: history of convergence in the $\LL{2}$-norm (vs. $h$) of the H-IIP (a-b), H-NIP (c-d) and H-SIP (e-f) schemes, respectively, on uniform triangular meshes with $-1\leq\delta\leq 2$. In the left images (a-c-e), $k=2$, and in the right images (b-d-f), $k=3$.}
\label{convergence2}
\end{figure}

\subsection{Test B: Influence of the parameter $\kappa_{\e{},\f{}}$}
\label{Numerical_infl_Kf}

In the second experiment, we analyze the behavior of the discretization method in the context of genuine anisotropic and heterogeneous properties. Then, the unit square $\domain$ is split into four subdomains $\domain_1 = [0,1/2]^2$, $\domain_2 = [1/2,1]\times[0,1/2]$, $\domain_3 = [1/2,1]^2$ and $\domain_4 = [0,1/2]\times[1/2,1]$, such that $\Omega\eqbydef\cup^{4}_{i=1}\Omega_i$ as illustrated in \figref{TestB-fig-descrip}. The exact solution on the whole domain $\domain$ is given by $u(x,y) = \sin(\pi x) \sin(\pi y)$, and the diffusivity tensor takes different values in each subregion:
\begin{align*}
    \bkap &= \begin{bmatrix}
    1 & 0 \\
    0 & \lambda
    \end{bmatrix} \quad &\textrm{for } (x,y) \in \domain_1, \, \domain_3,\\
    \bkap &= \begin{bmatrix}
    \inv{\lambda} & 0 \\
    0 & 1
    \end{bmatrix} \quad &\textrm{for } (x,y) \in \domain_2, \, \domain_4,
\end{align*}
where the parameter $\lambda>0$ simultaneously controls both the anisotropy and the medium heterogeneity.
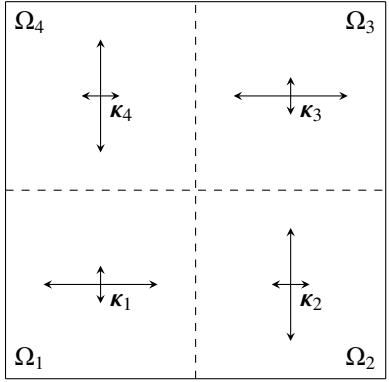
\begin{figure}[!ht]
    \centering
\begin{tikzpicture}[scale=0.50]
	    \draw (0,0) -- (10,0);
    	\draw (10,0) -- (10,10);
    	\draw (10,10) -- (0,10);	
    	\draw (0,10) -- (0,0);
    	
    	\draw[dashed] (0,5) -- (10,5);
    	\draw[dashed] (5,0) -- (5,10);
    	
    	\draw (0,0) node[above right] {$\Omega_1$};
    	\draw (10,0) node[above left] {$\Omega_2$};
    	\draw (10,10) node[below left] {$\Omega_3$};
    	\draw (0,10) node[below right] {$\Omega_4$};
    	
    	\draw[<-> , > = stealth] (1,2.5) -- (4,2.5);
    	\draw[<-> , > = stealth] (2.5,2) -- (2.5,3);
    	\draw (2.5,2.5) node[below right] {$\bkap_{1}$};
    	
    	\draw[<-> , > = stealth] (7,2.5) -- (8,2.5);
    	\draw[<-> , > = stealth] (7.5,1) -- (7.5,4);
    	\draw (7.5,2.5) node[below right] {$\bkap_{2}$};
    	
    	\draw[<-> , > = stealth] (6,7.5) -- (9,7.5);
    	\draw[<-> , > = stealth] (7.5,7) -- (7.5,8);
    	\draw (7.5,7.5) node[below right] {$\bkap_{3}$};
    	
    	\draw[<-> , > = stealth] (2,7.5) -- (3,7.5);
    	\draw[<-> , > = stealth] (2.5,6) -- (2.5,9);
    	\draw (2.5,7.5) node[below right] {$\bkap_{4}$};
	\end{tikzpicture}
	\caption{Description of test case B with genuine anisotropic and heterogeneous properties.}
    \label{TestB-fig-descrip}
\end{figure}
\noindent Here, we focus on the influence of the parameter $\kappa_{\e{},\f{}}$ on the robustness of the discretization method in the context of highly anisotropic and heterogeneous coefficients, and we choose $\lambda=10^{-3}$. In this context, the anisotropy and heterogeneity ratios are approximately $10^3$ and $10^6$, respectively. For the simulations, we consider a conforming triangular mesh ($h=1/32$) respecting the discontinuities of $\bkap$, we use piecewise linear approximations of the discrete variable $\du$, and we set $\delta=0$ in the definition of the penalty parameter \eqref{penalty_term}. Here, the comparisons are only graphical (\figref{Test B}). We depict the discrete solutions $\du$ obtained successively using $\kappa_{\e{},\f{}}\eqbydef 1$ (Case 1) and $\kappa_{\e{},\f{}}\eqbydef\vect{n}_{\e{},\f{}}\bkap_{\e{}}\vect{n}_{\e{},\f{}}$ (Case 2) for all variations of $\epsilon\in\{0,\pm 1\}$. In the first situation (Figures \ref{testBa}, \ref{testBc} and \ref{testBe}), the discrete solutions exhibit spurious oscillations and erratic behaviors, thus violating the discrete maximum principle. This can be easily explained by observing that the first formulation does not distinguish between the principal directions of the diffusivity tensor. Consequently, a misestimated penalty is applied in directions of low or high diffusivity. In the second situation (Figures \ref{testBb}, \ref{testBd} and \ref{testBf}), the jumps in diffusivity are better captured at the interfaces of discontinuities, and the discrete solutions are significantly more robust, \ie exhibit less erratic behavior.\par
\begin{figure}[!ht]
\centering
\subfloat[\label{testBa}]{\includegraphics[scale=0.16,trim = 445 100 325 110, clip=true]{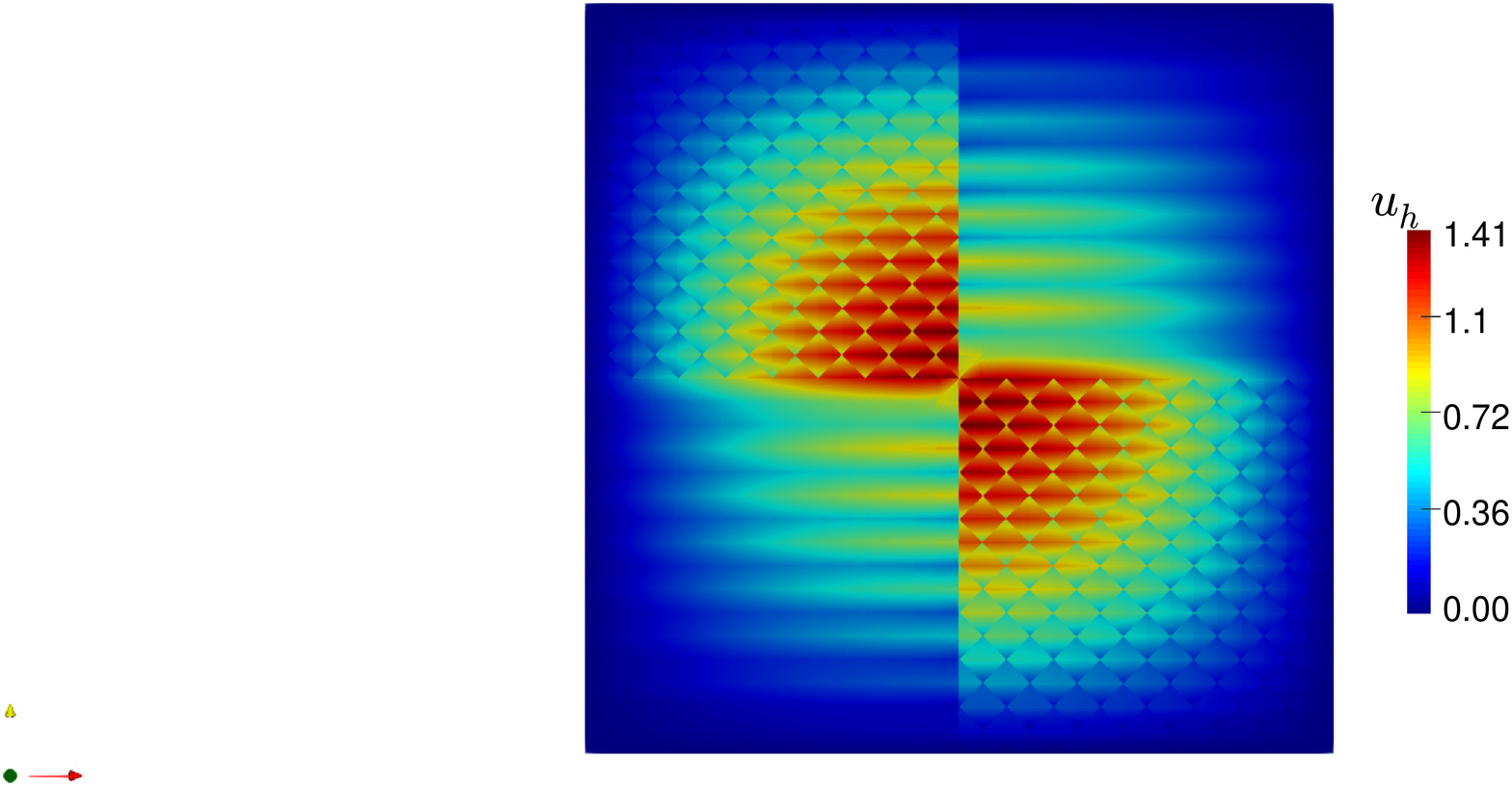}}
\subfloat[\label{testBb}]{\includegraphics[scale=0.16,trim = 445 100 325 110, clip=true]{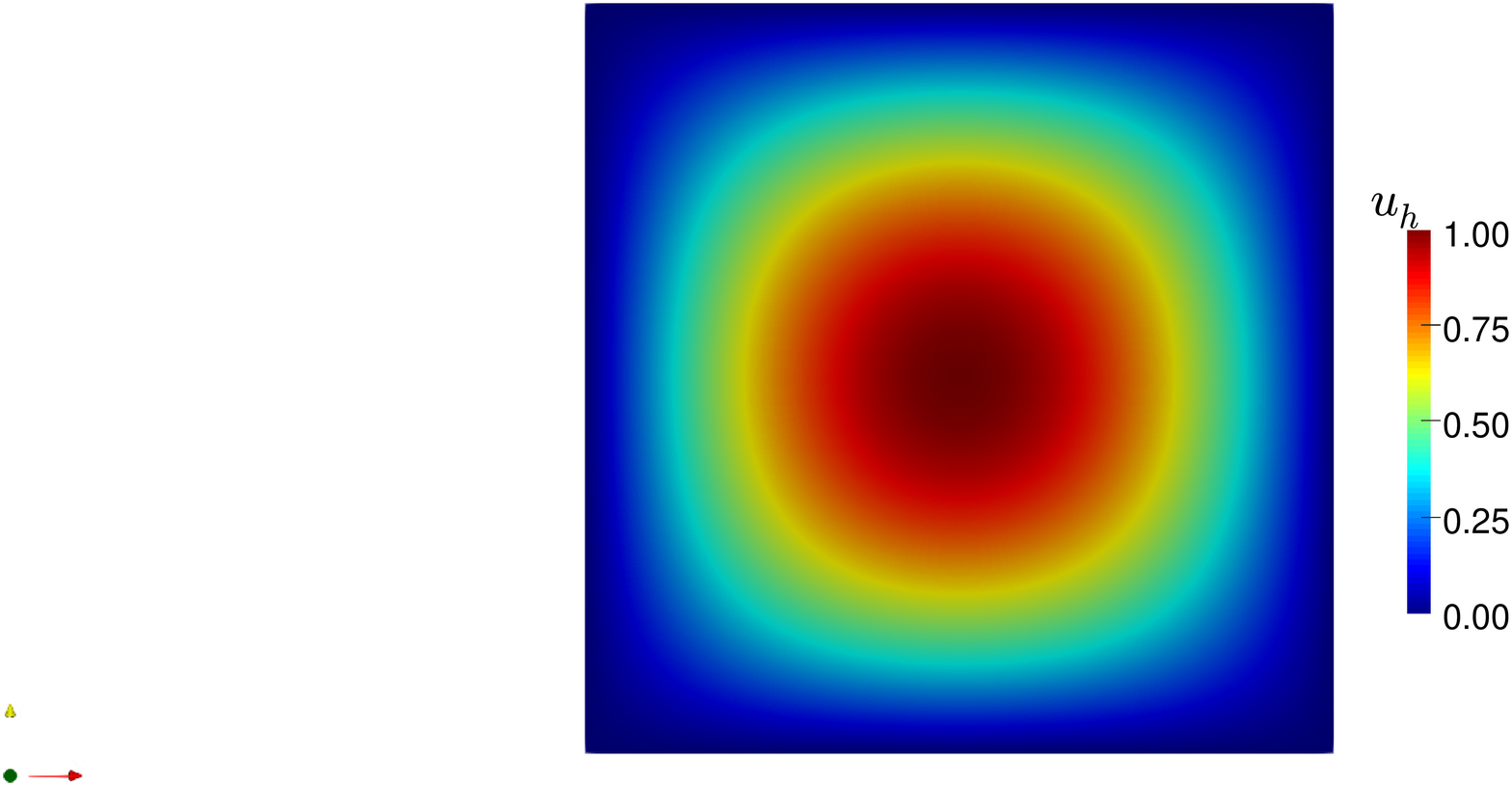}}\\
\subfloat[\label{testBc}]{\includegraphics[scale=0.16,trim = 445 100 325 110, clip=true]{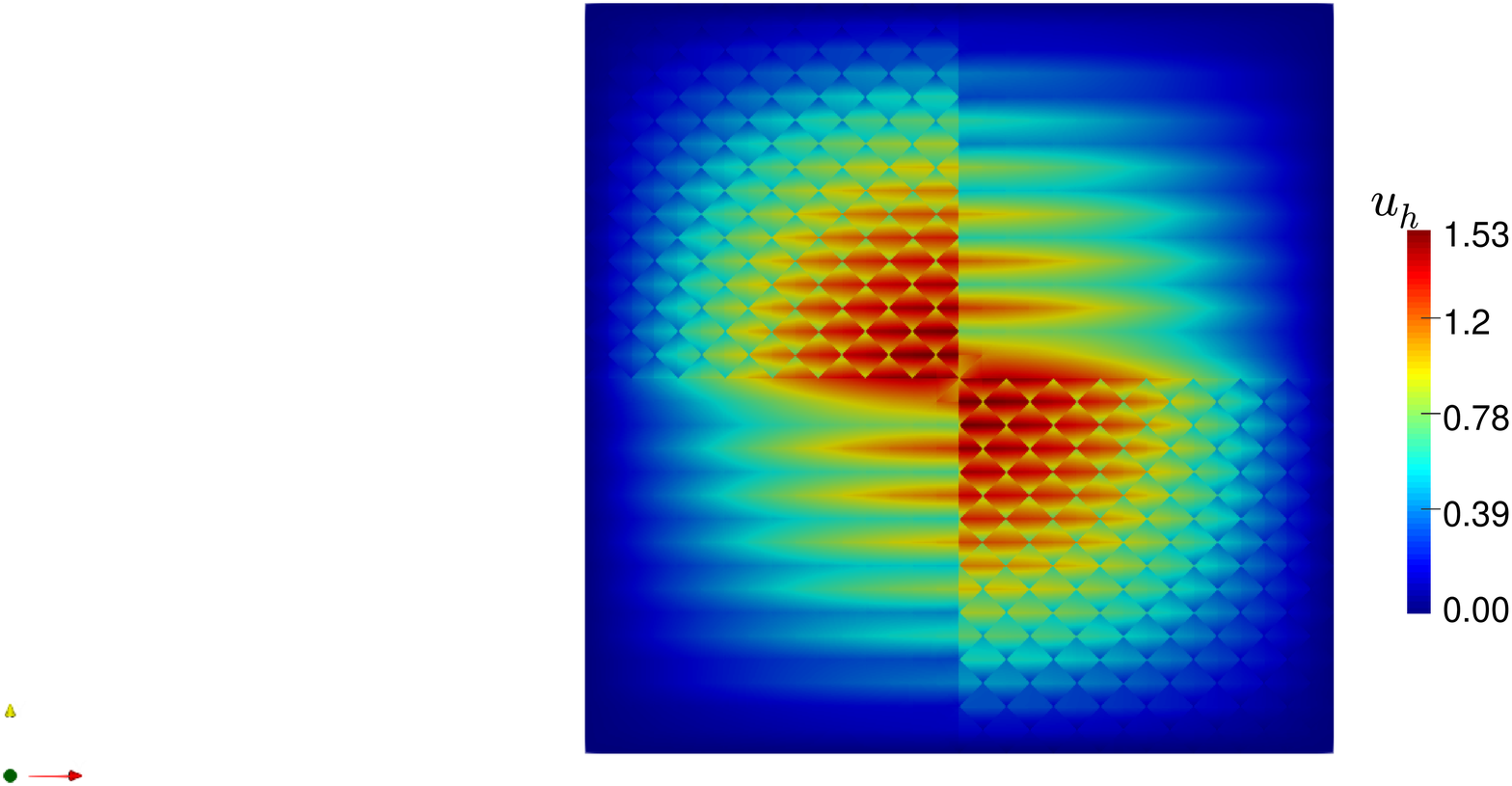}}
\subfloat[\label{testBd}]{\includegraphics[scale=0.16,trim = 445 100 325 110, clip=true]{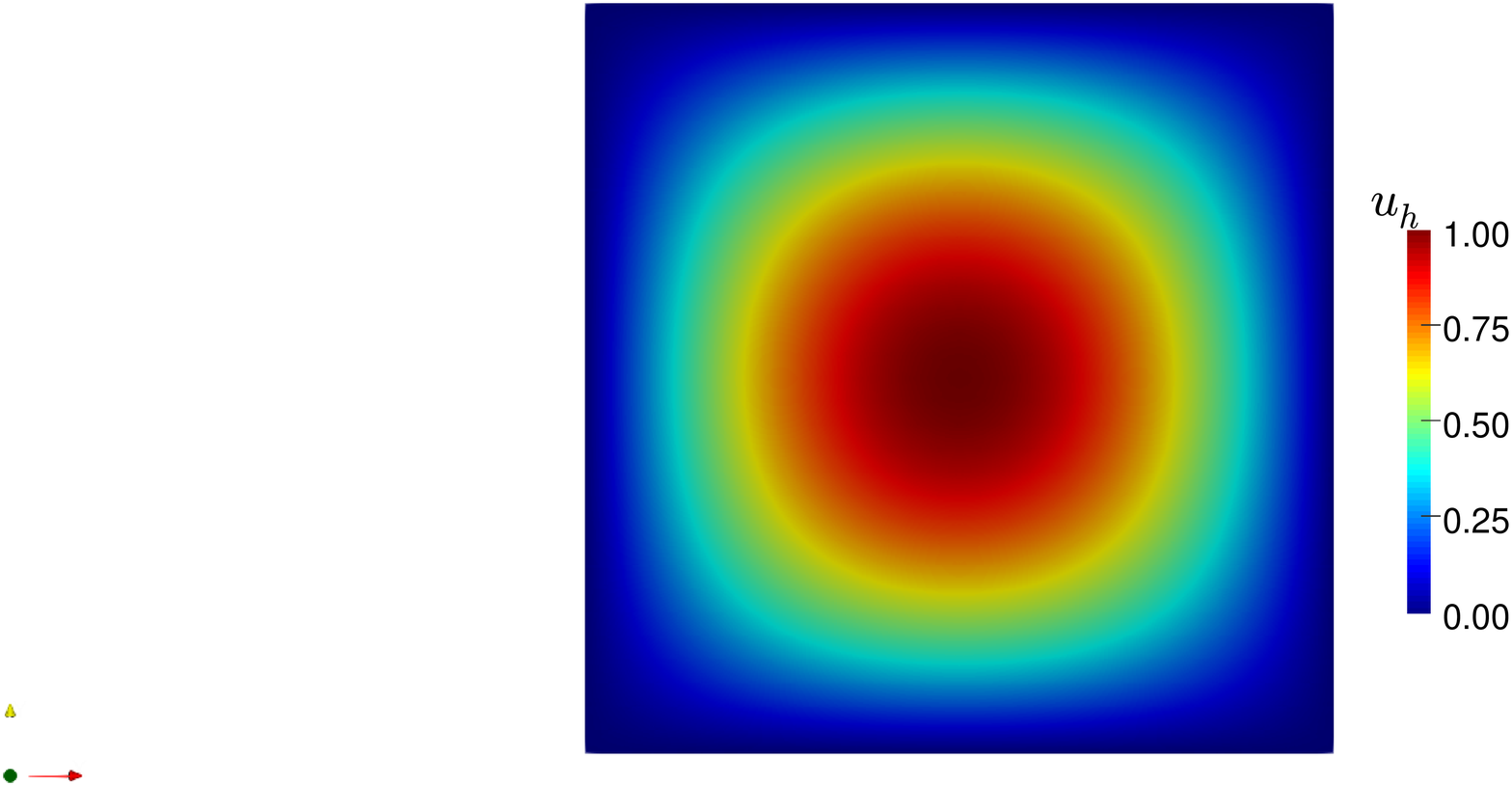}}\\
\subfloat[\label{testBe}]{\includegraphics[scale=0.16,trim = 445 100 325 110, clip=true]{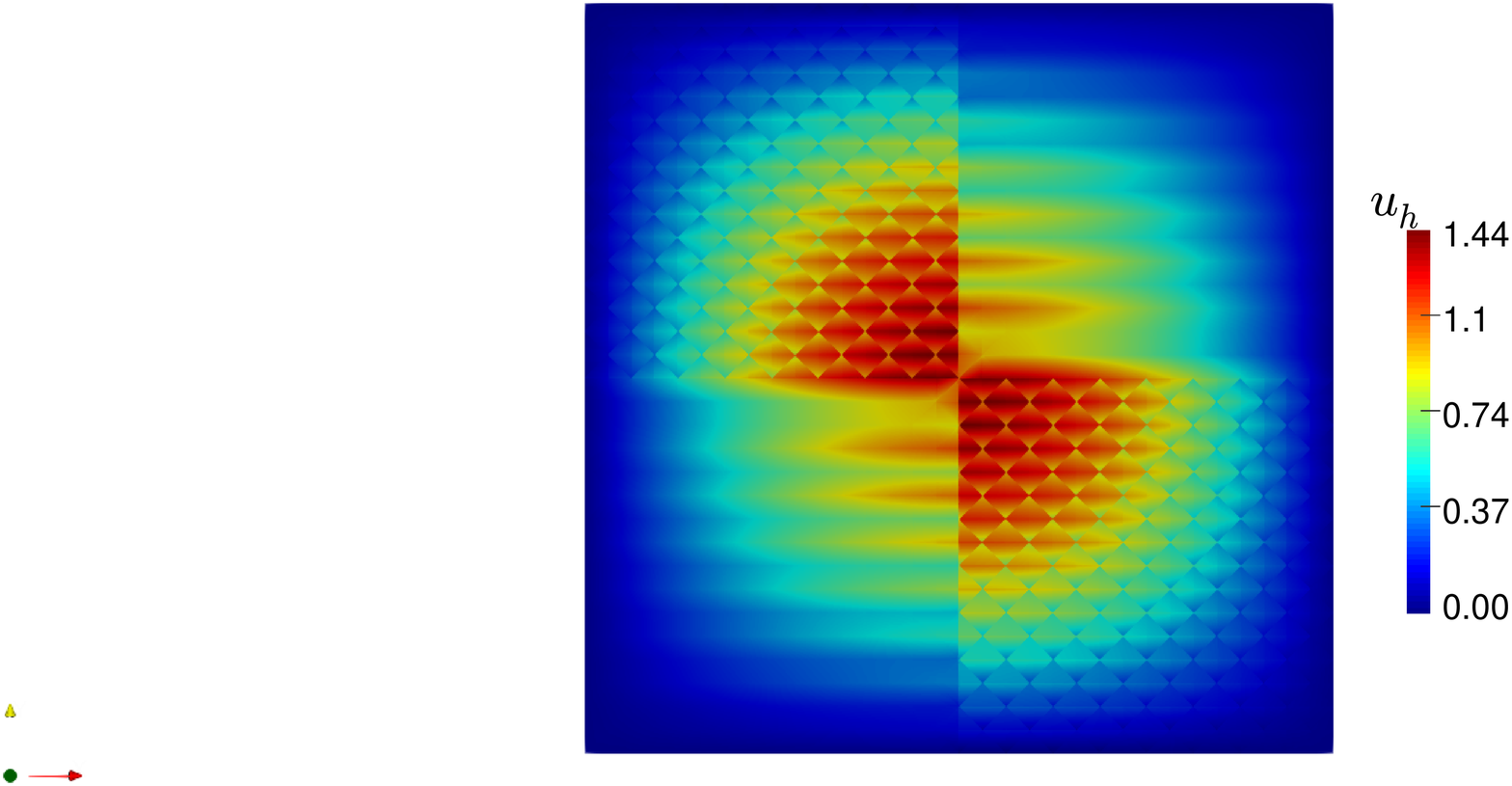}}
\subfloat[\label{testBf}]{\includegraphics[scale=0.16,trim = 445 100 325 110, clip=true]{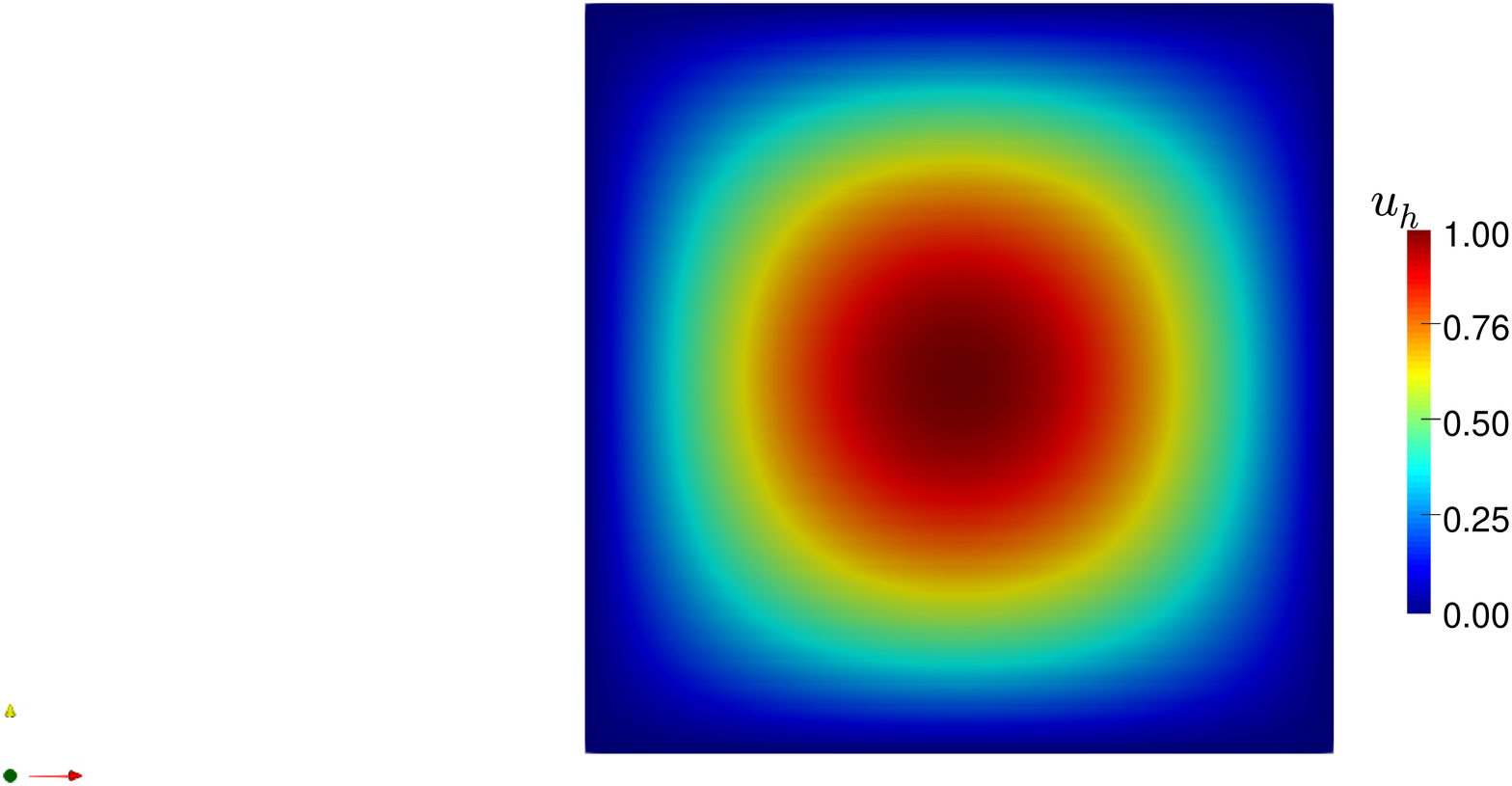}}
\caption{Test B: from the top to the bottom: representation of the discrete solution $\du$ obtained by the H-IIP (a-b), H-NIP (c-d) and H-SIP (e-f) schemes, respectively, on the structured triangular mesh ($h=1/32$). In the left images (a-c-e), the parameter $\kappa_{\e{},\f{}}$ in \eqref{penalty_term} is chosen as $\kappa_{\e{},\f{}}\eqbydef1$, and in the right images (b-d-f), $\kappa_{\e{},\f{}}\eqbydef\vect{n}_{\e{},\f{}}\bkap_{\e{}}\vect{n}_{\e{},\f{}}$.}
\label{Test B}
\end{figure}

\subsection{Test C: Influence of the parameter $\alpha_{0}$}
\label{Test C}

To conclude the sequence of numerical tests, we analyze the influence of the parameter $\alpha_0$ on the convergence of the H-SIP method for $\bkap$-orthogonal grids only. For simplicity, we consider the same test case as Test B, \eqref{Numerical_infl_Kf}, and we set two values of the parameter $\lambda$: (i) $\lambda=1$ for a homogeneous and isotropic media and (ii) $\lambda=0.1$  for a heterogeneous and anisotropic media. We plot the computed $\LL{2}$-error of the H-SIP method for a wide range of values of the parameter $\alpha_0$---\ie $1\leq\alpha_0\leq 6$---using a uniform square mesh ($h=1/32$). The analysis is done for polynomial degrees $1\leq k\leq 4$, but the results are presented for $k=1,2$ only. Analyzing \figref{gamma_dependency}, we observe that there exists an optimal value of the parameter $\alpha_0\eqbydef\alpha_{\textrm{opt}}$ that minimizes the $\LL{2}$-error of the scheme. In the context of $\bkap$-orthogonal grids, this optimal value ($\alpha_{\textrm{opt}}=2$) is insensitive to the mesh form, the mesh size $h$, the polynomial degree $k$, and the heterogeneity and/or anisotropy of the media $\lambda$. A history of the convergence of the H-SIP method using $\alpha_{\textrm{opt}}=2$ is then given in \figref{gamma_dependency}, and we note the surprising superconvergence of $\du$ $(k+2)$ in the discrete $\LL{2}$-norm obtained without any postprocessing. We emphasize that the \textit{superconvergence} property is not achieved for any triangular mesh or any value of the parameter $\epsilon\neq 1$, even using the optimal parameter $\alpha_{\textrm{opt}}$ in $\eqref{penalty_term}$.\par
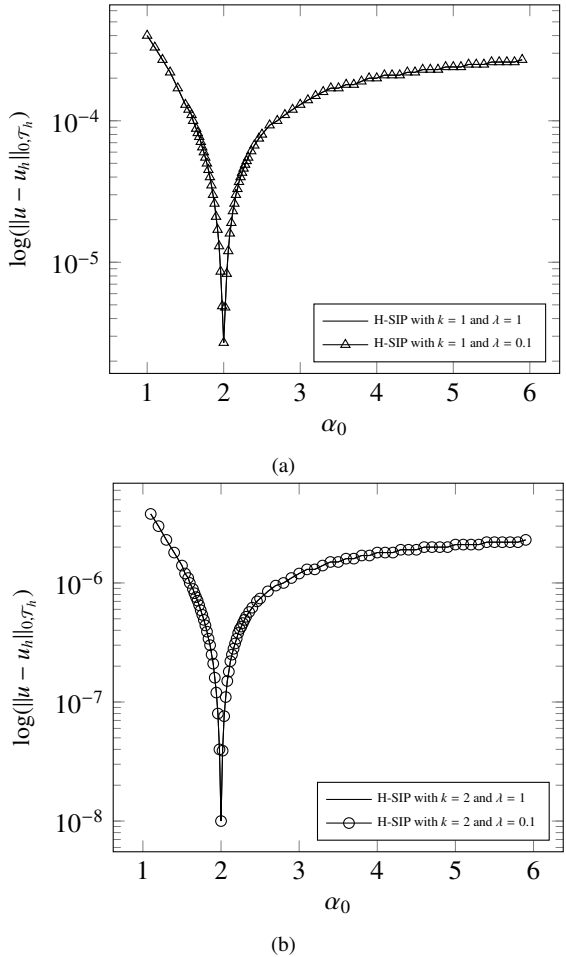
\begin{figure}[!ht]
\centering
\subfloat[]{\begin{tikzpicture}[scale=1.0]
	\begin{semilogyaxis}[
	     xlabel=$\alpha_0$, ylabel={\small{$\log(\norm{u-\du}{0}{\Th})$}}, legend cell align=left, legend pos = south east, font=\normalsize]
	\addplot[color=black] coordinates {
      (1.00,4.0e-04) (1.10,3.3e-04) (1.20,2.7e-04) (1.30,2.2e-04) (1.40,1.7e-04) (1.50,1.3e-04) (1.54,1.2e-04) (1.58,1.1e-04) (1.60,1.0e-04) (1.64,8.8e-05) (1.66,8.2e-05) (1.68,7.7e-05) (1.70,7.1e-05) (1.72,6.5e-05) (1.74,6.0e-05) (1.76,5.5e-05) (1.78,5.0e-05) (1.80,4.5e-05) (1.82,4.0e-05) (1.84,3.5e-05) (1.86,3.0e-05) (1.88,2.6e-05) (1.90,2.1e-05) (1.92,1.7e-05) (1.94,1.3e-05) (1.96,8.6e-06) (1.98,4.9e-06) (2.00,2.7e-06) (2.02,4.8e-06) (2.04,8.3e-06) (2.06,1.2e-05) (2.08,1.6e-05) (2.10,1.9e-05) (2.12,2.3e-05) (2.14,2.6e-05) (2.16,3.0e-05) (2.18,3.3e-05) (2.20,3.7e-05) (2.22,4.0e-05) (2.24,4.3e-05) (2.26,4.6e-05) (2.28,4.9e-05) (2.30,5.2e-05) (2.32,5.5e-05) (2.36,6.1e-05) (2.40,6.7e-05) (2.46,7.5e-05) (2.50,8.0e-05) (2.60,9.3e-05) (2.70,1.0e-04) (2.80,1.1e-04) (2.90,1.2e-04) (3.00,1.3e-04) (3.10,1.4e-04) (3.20,1.5e-04) (3.30,1.6e-04) (3.40,1.7e-04) (3.50,1.7e-04) (3.60,1.8e-04) (3.70,1.8e-04) (3.80,1.9e-04) (3.90,2.0e-04) (4.00,2.0e-04) (4.10,2.1e-04) (4.20,2.1e-04) (4.30,2.1e-04) (4.40,2.2e-04) (4.50,2.2e-04) (4.60,2.3e-04) (4.70,2.3e-04) (4.80,2.3e-04) (4.90,2.4e-04) (5.00,2.4e-04) (5.10,2.4e-04) (5.20,2.5e-04) (5.30,2.5e-04) (5.40,2.5e-04) (5.50,2.6e-04) (5.60,2.6e-04) (5.70,2.6e-04) (5.80,2.6e-04) (5.90,2.7e-04) 
	  };
	\addplot[color=black,mark=triangle] coordinates {
	 (1.00,4.0e-04) (1.10,3.3e-04) (1.20,2.7e-04) (1.30,2.2e-04) (1.40,1.7e-04) (1.50,1.3e-04) (1.54,1.2e-04) (1.58,1.1e-04) (1.60,1.0e-04) (1.64,8.8e-05) (1.66,8.2e-05) (1.68,7.7e-05) (1.70,7.1e-05) (1.72,6.5e-05) (1.74,6.0e-05) (1.76,5.5e-05) (1.78,5.0e-05) (1.80,4.5e-05) (1.82,4.0e-05) (1.84,3.5e-05) (1.86,3.0e-05) (1.88,2.6e-05) (1.90,2.1e-05) (1.92,1.7e-05) (1.94,1.3e-05) (1.96,8.6e-06) (1.98,4.9e-06) (2.00,2.7e-06) (2.02,4.8e-06) (2.04,8.3e-06) (2.06,1.2e-05) (2.08,1.6e-05) (2.10,1.9e-05) (2.12,2.3e-05) (2.14,2.6e-05) (2.16,3.0e-05) (2.18,3.3e-05) (2.20,3.7e-05) (2.22,4.0e-05) (2.24,4.3e-05) (2.26,4.6e-05) (2.28,4.9e-05) (2.30,5.2e-05) (2.32,5.5e-05) (2.36,6.1e-05) (2.40,6.7e-05) (2.46,7.5e-05) (2.50,8.0e-05) (2.60,9.3e-05) (2.70,1.0e-04) (2.80,1.1e-04) (2.90,1.2e-04) (3.00,1.3e-04) (3.10,1.4e-04) (3.20,1.5e-04) (3.30,1.6e-04) (3.40,1.7e-04) (3.50,1.7e-04) (3.60,1.8e-04) (3.70,1.8e-04) (3.80,1.9e-04) (3.90,2.0e-04) (4.00,2.0e-04) (4.10,2.1e-04) (4.20,2.1e-04) (4.30,2.1e-04) (4.40,2.2e-04) (4.50,2.2e-04) (4.60,2.3e-04) (4.70,2.3e-04) (4.80,2.3e-04) (4.90,2.4e-04) (5.00,2.4e-04) (5.10,2.4e-04) (5.20,2.5e-04) (5.30,2.5e-04) (5.40,2.5e-04) (5.50,2.6e-04) (5.60,2.6e-04) (5.70,2.6e-04) (5.80,2.6e-04) (5.90,2.7e-04) 
	 };
    \addlegendentry{\tiny{H-SIP with $k=1$ and $\lambda=1$}}
    \addlegendentry{\tiny{H-SIP with $k=1$ and $\lambda=0.1$}}
	\end{semilogyaxis}
\end{tikzpicture} } \quad
\subfloat[]{\begin{tikzpicture}[scale=1.0]
	\begin{semilogyaxis}[
	     xlabel=$\alpha_0$, ylabel=\small{$\log(\norm{u-\du}{0}{\Th})$}, legend cell align=left, legend pos = south east, font=\normalsize]
	 
	\addplot[color=black] coordinates {
	 (1.10,3.8e-06) (1.20,3.0e-06) (1.30,2.3e-06) (1.40,1.8e-06) (1.50,1.4e-06) (1.54,1.2e-06) (1.58,1.1e-06) (1.60,1.0e-06) (1.64,8.9e-07) (1.66,8.2e-07) (1.68,7.6e-07) (1.70,7.1e-07) (1.72,6.5e-07) (1.74,5.9e-07) (1.76,5.4e-07) (1.78,4.9e-07) (1.80,4.4e-07) (1.82,3.9e-07) (1.84,3.4e-07) (1.86,3.0e-07) (1.88,2.5e-07) (1.90,2.1e-07) (1.92,1.6e-07) (1.94,1.2e-07) (1.96,8.0e-08) (1.98,4.0e-08) (2.00,1.0e-08) (2.02,3.9e-08) (2.04,7.6e-08) (2.06,1.1e-07) (2.08,1.5e-07) (2.10,1.8e-07) (2.12,2.2e-07) (2.14,2.5e-07) (2.16,2.8e-07) (2.18,3.1e-07) (2.20,3.4e-07) (2.22,3.8e-07) (2.24,4.0e-07) (2.26,4.3e-07) (2.28,4.6e-07) (2.30,4.9e-07) (2.32,5.2e-07) (2.36,5.7e-07) (2.40,6.2e-07) (2.46,7.0e-07) (2.50,7.4e-07) (2.60,8.5e-07) (2.70,9.5e-07) (2.80,1.0e-06) (2.90,1.1e-06) (3.00,1.2e-06) (3.10,1.3e-06) (3.20,1.3e-06) (3.30,1.4e-06) (3.40,1.5e-06) (3.50,1.5e-06) (3.60,1.6e-06) (3.70,1.6e-06) (3.80,1.7e-06) (3.90,1.7e-06) (4.00,1.8e-06) (4.10,1.8e-06) (4.20,1.8e-06) (4.30,1.9e-06) (4.40,1.9e-06) (4.50,1.9e-06) (4.60,2.0e-06) (4.70,2.0e-06) (4.80,2.0e-06) (4.90,2.0e-06) (5.00,2.1e-06) (5.10,2.1e-06) (5.20,2.1e-06) (5.30,2.1e-06) (5.40,2.2e-06) (5.50,2.2e-06) (5.60,2.2e-06) (5.70,2.2e-06) (5.80,2.2e-06) (5.90,2.3e-06) 
	  };
	
	\addplot[color=black,mark=o] coordinates {
	 (1.10,3.8e-06) (1.20,3.0e-06) (1.30,2.3e-06) (1.40,1.8e-06) (1.50,1.4e-06) (1.54,1.2e-06) (1.58,1.1e-06) (1.60,1.0e-06) (1.64,8.9e-07) (1.66,8.2e-07) (1.68,7.6e-07) (1.70,7.1e-07) (1.72,6.5e-07) (1.74,5.9e-07) (1.76,5.4e-07) (1.78,4.9e-07) (1.80,4.4e-07) (1.82,3.9e-07) (1.84,3.4e-07) (1.86,3.0e-07) (1.88,2.5e-07) (1.90,2.1e-07) (1.92,1.6e-07) (1.94,1.2e-07) (1.96,8.0e-08) (1.98,4.0e-08) (2.00,1.0e-08) (2.02,3.9e-08) (2.04,7.6e-08) (2.06,1.1e-07) (2.08,1.5e-07) (2.10,1.8e-07) (2.12,2.2e-07) (2.14,2.5e-07) (2.16,2.8e-07) (2.18,3.1e-07) (2.20,3.4e-07) (2.22,3.8e-07) (2.24,4.1e-07) (2.26,4.3e-07) (2.28,4.6e-07) (2.30,4.9e-07) (2.32,5.2e-07) (2.36,5.7e-07) (2.40,6.2e-07) (2.46,7.0e-07) (2.50,7.4e-07) (2.60,8.5e-07) (2.70,9.5e-07) (2.80,1.0e-06) (2.90,1.1e-06) (3.00,1.2e-06) (3.10,1.3e-06) (3.20,1.3e-06) (3.30,1.4e-06) (3.40,1.5e-06) (3.50,1.5e-06) (3.60,1.6e-06) (3.70,1.6e-06) (3.80,1.7e-06) (3.90,1.7e-06) (4.00,1.8e-06) (4.10,1.8e-06) (4.20,1.8e-06) (4.30,1.9e-06) (4.40,1.9e-06) (4.50,1.9e-06) (4.60,2.0e-06) (4.70,2.0e-06) (4.80,2.0e-06) (4.90,2.0e-06) (5.00,2.1e-06) (5.10,2.1e-06) (5.20,2.1e-06) (5.30,2.1e-06) (5.40,2.2e-06) (5.50,2.2e-06) (5.60,2.2e-06) (5.70,2.2e-06) (5.80,2.2e-06) (5.90,2.3e-06) 
	 };
    \addlegendentry{\tiny{H-SIP with $k=2$ and $\lambda=1$}}
    \addlegendentry{\tiny{H-SIP with $k=2$ and $\lambda=0.1$}}
	\end{semilogyaxis}
\end{tikzpicture}}
\caption{Test C: the $\LL{2}$-error of the H-SIP method vs. $\alpha_0$ for a uniform square mesh using piecewise linear (a) and quadratic (b) approximations.}
\label{gamma_dependency}
\end{figure}
\begin{table}[!ht]
\centering
\small
\caption{Test C: history of convergence $\norm{u-\du}{0}{\Th}$ (vs. $h$) of the H-SIP method using the optimal parameter $\alpha_{\textrm{opt}}$ on uniform square meshes}
\label{tableA}
\begin{tabular}{|c|cccc|}
\hline
& \multicolumn{4}{c|}{H-SIP ($k=1$)} \\
& \multicolumn{2}{c}{$\lambda=1$} & \multicolumn{2}{c|}{$\lambda=0.1$} \\
\hline
$h^{-1}$ & $\norm{u-\du}{0}{\Th}$ & ECR & $\norm{u-\du}{0}{\Th}$ & ECR \\
\hline
$8$ & $1.7e-04$ & -- & $1.7e-04$ & -- \\
$16$ & $2.1e-05$ & \textcolor{red}{$3.00$} & $2.1e-05$ & \textcolor{red}{$3.00$} \\
$32$ & $2.7e-06$ & \textcolor{red}{$3.00$} & $2.7e-06$ & \textcolor{red}{$3.00$} \\
$64$ & $3.4e-07$ & \textcolor{red}{$3.00$} & $3.4e-07$ & \textcolor{red}{$3.00$} \\
\hline
& \multicolumn{4}{c|}{H-SIP ($k=2$)} \\
\hline
$8$  & $2.6e-06$ & -- & $2.6e-06$ & -- \\
$16$ & $1.6e-07$ & \textcolor{red}{$3.99$} & $1.6e-07$ & \textcolor{red}{$3.99$} \\
$32$ & $1.0e-08$ & \textcolor{red}{$4.00$} & $1.0e-08$ & \textcolor{red}{$4.00$} \\
$64$ & $6.4e-10$ & \textcolor{red}{$4.00$} & $6.4e-10$ & \textcolor{red}{$4.00$} \\
\hline
\end{tabular}
\end{table}

\section{Conclusion}

We derive improved \textit{a priori} error estimates of families of hybridizable interior penalty discontinuous Galerkin methods using a variable penalty to solve highly anisotropic diffusion problems. The convergence analysis highlights the $h^{\delta}$-dependency of the coercivity condition and the boundedness requirement that strongly impacts the derived error estimates in terms of both energy- and $\LL{2}$-norms. The optimal convergence of the energy-norm is proven for any penalty parameter $\delta\geq 0$ and $\epsilon\in\{0,\pm 1\}$. The situation is somewhat different in $\LL{2}$, and distinctive features can be found between the three schemes. Indeed, the symmetric method theoretically converges optimally if $\delta\geq 0$, and non-symmetric variants converge only if $\delta\geq 2$ independently of the polynomial parity. All of these estimates are corroborated by numerical evidence. Notably, the superconvergence of the H-SIP scheme is achieved for $\bkap$-orthogonal grids without any postprocessing but only if an appropriate $\alpha_0$ is selected.

\section*{Acknowledgments}
By convention, the names of the authors are listed in alphabetical order. The third author is grateful to Sander Rhebergen for his invitation to the Department of Applied Mathematics at the University of Waterloo (UW) in April 2019. He also would like to thank B\'eatrice Rivi\`ere at Rice University for her insightful suggestions and remarks concerning \textit{a priori} error estimates. Our fruitful discussions of (hybridizable) interior penalty methods using superpenalties were the source of inspiration of the present work.

%
%


\bibliography{references.bib}

\begin{thebibliography}{24}
\expandafter\ifx\csname natexlab\endcsname\relax\def\natexlab#1{#1}\fi
\providecommand{\url}[1]{\texttt{#1}}
\providecommand{\href}[2]{#2}
\providecommand{\path}[1]{#1}
\providecommand{\DOIprefix}{doi:}
\providecommand{\ArXivprefix}{arXiv:}
\providecommand{\URLprefix}{URL: }
\providecommand{\Pubmedprefix}{pmid:}
\providecommand{\doi}[1]{\href{http://dx.doi.org/#1}{\path{#1}}}
\providecommand{\Pubmed}[1]{\href{pmid:#1}{\path{#1}}}
\providecommand{\bibinfo}[2]{#2}
\ifx\xfnm\relax \def\xfnm[#1]{\unskip,\space#1}\fi
\bibitem[{Cockburn et~al.(2009)Cockburn, Gopalakrishnan, and
  Lazarov}]{Cockburn09unified}
\bibinfo{author}{B.~Cockburn}, \bibinfo{author}{J.~Gopalakrishnan},
  \bibinfo{author}{R.~Lazarov},
\newblock \bibinfo{title}{Unified hybridization of discontinuous galerkin,
  mixed, and continuous galerkin methods for second order elliptic problems},
\newblock \bibinfo{journal}{SIAM Journal on Numerical Analysis}
  \bibinfo{volume}{47} (\bibinfo{year}{2009}) \bibinfo{pages}{1319--1365}.
  \DOIprefix\doi{https://doi.org/10.1137/070706616}.
\bibitem[{Egger and Sch{\"o}berl(2009)}]{egger2009mixed}
\bibinfo{author}{H.~Egger}, \bibinfo{author}{J.~Sch{\"o}berl},
\newblock \bibinfo{title}{A mixed-hybrid-discontinuous galerkin finite element
  method for convection-diffusion problems},
\newblock \bibinfo{journal}{IMA J. Numer. Anal} \bibinfo{volume}{30}
  (\bibinfo{year}{2009}) \bibinfo{pages}{1--2}.
  \DOIprefix\doi{https://doi.org/10.1093/imanum/drn083}.
\bibitem[{Fabien et~al.(2020)Fabien, Knepley, and Riviere}]{Fabien2020high}
\bibinfo{author}{M.~S. Fabien}, \bibinfo{author}{M.~Knepley},
  \bibinfo{author}{B.~Riviere},
\newblock \bibinfo{title}{A high order hybridizable discontinuous galerkin
  method for incompressible miscible displacement in heterogeneous media},
\newblock \bibinfo{journal}{Results in Applied Mathematics}
  (\bibinfo{year}{2020}) \bibinfo{pages}{100089}.
  \DOIprefix\doi{https://doi.org/10.1016/j.rinam.2019.100089}.
\bibitem[{Etangsale et~al.(2021)Etangsale, Fahs, Fontaine, and
  Isa-Abadi}]{etangsale:hal-03247309}
\bibinfo{author}{G.~Etangsale}, \bibinfo{author}{M.~Fahs},
  \bibinfo{author}{V.~Fontaine}, \bibinfo{author}{A.~R. Isa-Abadi},
  \bibinfo{title}{{Families of hybridizable interior penalty discontinuous
  Galerkin methods for degenerate advection-diffusion-reaction problems}},
  \bibinfo{year}{2021}. \URLprefix
  \url{https://hal.archives-ouvertes.fr/hal-03247309}, \bibinfo{note}{22 pages,
  (preprint submitted to the Journal of Computational Physics)}.
\bibitem[{Kirk and Rhebergen(2019)}]{kirk2019analysis}
\bibinfo{author}{K.~L. Kirk}, \bibinfo{author}{S.~Rhebergen},
\newblock \bibinfo{title}{Analysis of a pressure-robust hybridized
  discontinuous galerkin method for the stationary navier--stokes equations},
\newblock \bibinfo{journal}{Journal of Scientific Computing}
  \bibinfo{volume}{81} (\bibinfo{year}{2019}) \bibinfo{pages}{881--897}.
  \DOIprefix\doi{https://doi.org/10.1007/s10915-019-01040-y}.
\bibitem[{Sevilla(2019)}]{SEVILLA201969}
\bibinfo{author}{R.~Sevilla},
\newblock \bibinfo{title}{Hdg-nefem for two dimensional linear elasticity},
\newblock \bibinfo{journal}{Computers \& Structures} \bibinfo{volume}{220}
  (\bibinfo{year}{2019}) \bibinfo{pages}{69--80}.
  \DOIprefix\doi{https://doi.org/10.1016/j.compstruc.2019.05.005}.
\bibitem[{Sánchez et~al.(2017)Sánchez, Ciuca, Nguyen, Peraire, and
  Cockburn}]{SANCHEZ2017951}
\bibinfo{author}{M.~Sánchez}, \bibinfo{author}{C.~Ciuca},
  \bibinfo{author}{N.~Nguyen}, \bibinfo{author}{J.~Peraire},
  \bibinfo{author}{B.~Cockburn},
\newblock \bibinfo{title}{Symplectic hamiltonian hdg methods for wave
  propagation phenomena},
\newblock \bibinfo{journal}{Journal of Computational Physics}
  \bibinfo{volume}{350} (\bibinfo{year}{2017}) \bibinfo{pages}{951--973}.
  \DOIprefix\doi{https://doi.org/10.1016/j.jcp.2017.09.010}.
\bibitem[{Lee et~al.(2019)Lee, Shannon, Bui-Thanh, and
  Shadid}]{Lee2019AnalysisOA}
\bibinfo{author}{J.~J. Lee}, \bibinfo{author}{S.~Shannon},
  \bibinfo{author}{T.~Bui-Thanh}, \bibinfo{author}{J.~N. Shadid},
\newblock \bibinfo{title}{Analysis of an hdg method for linearized
  incompressible resistive mhd equations},
\newblock \bibinfo{journal}{SIAM J. Numer. Anal.} \bibinfo{volume}{57}
  (\bibinfo{year}{2019}) \bibinfo{pages}{1697--1722}.
  \DOIprefix\doi{https://doi.org/10.1137/18M1166729}.
\bibitem[{Arnold et~al.(2002)Arnold, Brezzi, Cockburn, and
  Marini}]{arnold2002unified}
\bibinfo{author}{D.~N. Arnold}, \bibinfo{author}{F.~Brezzi},
  \bibinfo{author}{B.~Cockburn}, \bibinfo{author}{L.~D. Marini},
\newblock \bibinfo{title}{Unified analysis of discontinuous galerkin methods
  for elliptic problems},
\newblock \bibinfo{journal}{SIAM journal on numerical analysis}
  \bibinfo{volume}{39} (\bibinfo{year}{2002}) \bibinfo{pages}{1749--1779}.
\bibitem[{Kirby et~al.(2012)Kirby, Sherwin, and Cockburn}]{kirby2012cg}
\bibinfo{author}{R.~M. Kirby}, \bibinfo{author}{S.~J. Sherwin},
  \bibinfo{author}{B.~Cockburn},
\newblock \bibinfo{title}{To cg or to hdg: a comparative study},
\newblock \bibinfo{journal}{Journal of Scientific Computing}
  \bibinfo{volume}{51} (\bibinfo{year}{2012}) \bibinfo{pages}{183--212}.
  \DOIprefix\doi{https://doi.org/10.1007/s10915-011-9501-7}.
\bibitem[{Nguyen et~al.(2009)Nguyen, Peraire, and
  Cockburn}]{nguyen2009implicit}
\bibinfo{author}{N.~C. Nguyen}, \bibinfo{author}{J.~Peraire},
  \bibinfo{author}{B.~Cockburn},
\newblock \bibinfo{title}{An implicit high-order hybridizable discontinuous
  galerkin method for linear convection-diffusion equations},
\newblock \bibinfo{journal}{Journal of Computational Physics}
  \bibinfo{volume}{228} (\bibinfo{year}{2009}) \bibinfo{pages}{3232--3254}.
  \DOIprefix\doi{https://doi.org/10.1016/j.jcp.2009.01.030}.
\bibitem[{Lehrenfeld(2010)}]{lehrenfeld2010hybrid}
\bibinfo{author}{C.~Lehrenfeld},
\newblock \bibinfo{title}{Hybrid discontinuous galerkin methods for solving
  incompressible flow problems},
\newblock \bibinfo{journal}{Rheinisch-Westfalischen Technischen Hochschule
  Aachen}  (\bibinfo{year}{2010}) \bibinfo{pages}{111}.
\bibitem[{Fabien et~al.(2019)Fabien, Knepley, and Riviere}]{FabienNM2019}
\bibinfo{author}{M.~S. Fabien}, \bibinfo{author}{M.~G. Knepley},
  \bibinfo{author}{B.~M. Riviere},
\newblock \bibinfo{title}{Families of interior penalty hybridizable
  discontinuous galerkin methods for second order elliptic problems},
\newblock \bibinfo{journal}{Journal of Numerical Mathematics}
  (\bibinfo{year}{2019}).
  \DOIprefix\doi{https://doi.org/10.1515/jnma-2019-0027}.
\bibitem[{Dijoux et~al.(2019)Dijoux, Fontaine, and Mara}]{DijouxA2M19}
\bibinfo{author}{L.~Dijoux}, \bibinfo{author}{V.~Fontaine},
  \bibinfo{author}{T.~A. Mara},
\newblock \bibinfo{title}{A projective hybridizable discontinuous galerkin
  mixed method for second-order diffusion problems},
\newblock \bibinfo{journal}{Applied Mathematical Modelling}
  \bibinfo{volume}{75} (\bibinfo{year}{2019}) \bibinfo{pages}{663--677}.
  \DOIprefix\doi{https://doi.org/10.1016/j.apm.2019.05.054}.
\bibitem[{Wells(2011)}]{wells2011analysis}
\bibinfo{author}{G.~N. Wells},
\newblock \bibinfo{title}{Analysis of an interface stabilized finite element
  method: the advection-diffusion-reaction equation},
\newblock \bibinfo{journal}{SIAM Journal on Numerical Analysis}
  \bibinfo{volume}{49} (\bibinfo{year}{2011}) \bibinfo{pages}{87--109}.
  \DOIprefix\doi{https://doi.org/10.1137/090775464}.
\bibitem[{Arnold(1982)}]{arnold1982interior}
\bibinfo{author}{D.~N. Arnold},
\newblock \bibinfo{title}{An interior penalty finite element method with
  discontinuous elements},
\newblock \bibinfo{journal}{SIAM journal on numerical analysis}
  \bibinfo{volume}{19} (\bibinfo{year}{1982}) \bibinfo{pages}{742--760}.
\bibitem[{Riviere(2008)}]{Riviere08bookDG}
\bibinfo{author}{B.~Riviere}, \bibinfo{title}{Discontinuous Galerkin methods
  for solving elliptic and parabolic equations: theory and implementation},
  \bibinfo{publisher}{SIAM}, \bibinfo{year}{2008}.
\bibitem[{Di~Pietro and Ern(2011)}]{dpe2011mathematical}
\bibinfo{author}{D.~A. Di~Pietro}, \bibinfo{author}{A.~Ern},
  \bibinfo{title}{Mathematical aspects of discontinuous Galerkin methods},
  volume~\bibinfo{volume}{69}, \bibinfo{publisher}{Springer Science \& Business
  Media}, \bibinfo{year}{2011}.
\bibitem[{Oikawa(2017)}]{oikawa2017hdg}
\bibinfo{author}{I.~Oikawa},
\newblock \bibinfo{title}{Hdg methods for second-order elliptic problems
  (numerical analysis: New developments for elucidating interdisciplinary
  problems ii)},
\newblock \bibinfo{journal}{RIMS Kokyuroku} \bibinfo{volume}{2037}
  (\bibinfo{year}{2017}) \bibinfo{pages}{61--74}.
\bibitem[{Rivi\`ere et~al.(1998)Rivi\`ere, Wheeler, and
  Girault}]{Riviere98IPDG}
\bibinfo{author}{B.~Rivi\`ere}, \bibinfo{author}{M.~Wheeler},
  \bibinfo{author}{V.~Girault},
\newblock \bibinfo{title}{Improved energy estimates for interior penalty,
  constrained and discontinuous galerkin methods for elliptic problems},
\newblock \bibinfo{journal}{Computational Geosciences} \bibinfo{volume}{3}
  (\bibinfo{year}{1998}) \bibinfo{pages}{337--360}.
  \DOIprefix\doi{https://doi.org/10.1023/A:1011591328604}.
\bibitem[{Guzm{\'a}n and Rivi\`ere(2009)}]{Guzman09JSC}
\bibinfo{author}{J.~Guzm{\'a}n}, \bibinfo{author}{B.~Rivi\`ere},
\newblock \bibinfo{title}{Sub-optimal convergence of non-symmetric
  discontinuous galerkin methods for odd polynomial approximations},
\newblock \bibinfo{journal}{Journal of Scientific Computing}
  \bibinfo{volume}{40} (\bibinfo{year}{2009}) \bibinfo{pages}{273--280}.
  \DOIprefix\doi{https://doi.org/10.1007/s10915-008-9255-z}.
\bibitem[{Ciarlet(1991)}]{ciarlet1991basic}
\bibinfo{author}{P.~G. Ciarlet},
\newblock \bibinfo{title}{Basic error estimates for elliptic problems},
\newblock \bibinfo{journal}{Handbook of Numerical Analysis} \bibinfo{volume}{2}
  (\bibinfo{year}{1991}) \bibinfo{pages}{17--351}.
\bibitem[{Ern et~al.(2009)Ern, Stephansen, and Zunino}]{ern2009discontinuous}
\bibinfo{author}{A.~Ern}, \bibinfo{author}{A.~F. Stephansen},
  \bibinfo{author}{P.~Zunino},
\newblock \bibinfo{title}{A discontinuous galerkin method with weighted
  averages for advection--diffusion equations with locally small and
  anisotropic diffusivity},
\newblock \bibinfo{journal}{IMA Journal of Numerical Analysis}
  \bibinfo{volume}{29} (\bibinfo{year}{2009}) \bibinfo{pages}{235--256}.
  \DOIprefix\doi{https://doi.org/10.1093/imanum/drm050}.
\bibitem[{Sch{\"o}berl(2014)}]{Schoberl2014c++}
\bibinfo{author}{J.~Sch{\"o}berl},
\newblock \bibinfo{title}{C++ 11 implementation of finite elements in ngsolve},
\newblock \bibinfo{journal}{Institute for Analysis and Scientific Computing,
  Vienna University of Technology}  (\bibinfo{year}{2014}).

\end{thebibliography}

%
%

\end{document}